\documentclass[12pt]{amsart}

\usepackage{amscd}

\usepackage{amsmath,amsfonts,amsthm,epsfig,latexsym,graphicx,amssymb}
\usepackage[english]{babel}

\newtheorem{coro}{Corollary}[section]
\newtheorem{defi}{Definition}[section]
\newtheorem{prop}{Proposition}[section]

\newtheorem{theo}{Theorem}
\newtheorem{lemm}{Lemma}[section]
\newtheorem{ques}{Question}[section]

\newtheorem{prob}{Problem}
\newtheorem{rema}{Remark}

\newtheorem{clai}{Claim}

\newtheorem*{theorem*}{Theorem}

\def\R{I\kern -0.37 em R}
\def\N{I\kern -0.37 em N}
\def\Z{I\kern -0.37 em Z}

\def\supess_#1{\mathop{\rm supess}\limits_{#1}}
\def\infess_#1{\mathop{\rm infess}\limits_{#1}}

\def\DD{{\mathbb D}}

 \def\NN{{\mathbb N}} 

 \def\RR{{\mathbb R}} 
\def\TT{{\mathbb T}}

 \def\ZZ{{\mathbb Z}}

\def\La{\Lambda}
\def\De{\Delta}

\def\Ga{\Gamma}

\def\cA{{\mathcal A}}  \def\cG{{\mathcal G}}  \def\cS{{\mathcal S}} \def\cY{{\mathcal Y}}
   \def\cN{{\mathcal N}}  
  \def\cI{{\mathcal I}}  \def\cU{{\mathcal U}}
\def\cD{{\mathcal D}}   \def\cP{{\mathcal P}} 
\def\cE{{\mathcal E}}    
\def\cF{{\mathcal F}}   \def\cR{{\mathcal R}} \def\cX{{\mathcal X}}

\title[Anosov flows obtained by surgeries.]{\bf Anosov flows on $3$-manifolds: the surgeries and the foliations 
}

\begin{document}


\maketitle
\date{}

\author{Christian Bonatti, Ioannis Iakovoglou}
\begin{abstract}To any Anosov flow $X$ on a $3$-manifold \cite{Fe1} associated a bi-foliated plane (a plane endowed with two transverse foliations $F^s$ and $F^u$) which reflects the normal structure of the flow endowed with the center-stable and center unstable foliations. A flow is \emph{$\RR$-covered} if $F^s$ (or equivalently $F^u$) 
is trivial. 

On the other hand, from one Anosov flow one can build infinitely many others by Dehn-Goodman-Fried surgeries. This paper investigates how these surgeries modify the bi-foliated plane. 

We first noticed that surgeries along some specific periodic orbits do not modify the bi-foliated plane: for instance,
\begin{itemize}\item surgeries on families of orbits corresponding to disjoint simple closed geodesics
do not affect the bi-foliated plane associated to the geodesic flow of a hyperbolic surface (Theorem~\ref{t.geodesic});
\item \cite{Fe2} associates a (non-empty) finite family of periodic orbits, called \emph{pivots}, to any non-$\RR$-covered Anosov flow. Surgeries on pivots do not affect the branching structure of the bi-foliated plane  (Theorem~\ref{t.pivot0})
\end{itemize}

We consider the set ${\cS}urg(A)$ of Anosov flows obtained by Dehn-Goodman-Fried surgery 
from the suspension flows of Anosov automorphisms $A\in SL(2,\ZZ)$ of the torus $\TT^2$. 

Every such surgery is associated to a finite set of couples $\lbrace(\gamma_i,m_i)\rbrace_{i}$, where the $\gamma_i$ are periodic orbits and the $m_i$ integers. When all the $m_i$ have the same sign, Fenley proved that the induced Anosov flow is $\RR$-covered and twisted according to the sign of the surgery. We analyse here the case where the surgeries are positive on a finite set of periodic points $X$ and negative on another set $Y$. In particular we build non-$\RR$-covered Anosov flows on hyperbolic manifolds.

Among other results, we show that given any flow $X\in {\cS}urg(A)$ :
\begin{itemize}
 \item there exists $\epsilon>0$ such that for every $\varepsilon$-dense periodic orbit $\gamma$, 
every flow obtained from $X$ by a non trivial surgery along $\gamma$ is $\RR$-covered (Theorem~\ref{t.epsilon}). 
\item there exist periodic orbits $\gamma_+,\gamma_-$ such that every flow obtained from $X$ by surgeries with distinct signs on $\gamma_+$ and $\gamma_-$ is non-$\RR$-covered (Theorem~\ref{t.nonRcovered}). 
\end{itemize}

\end{abstract}
\vskip 5mm
\footnotesize{

\textbf{Codes AMS: 37D20-37D40-57M10}}


\section{Introduction}
\subsection{General setting}
In this paper we consider Anosov flows on closed $3$-manifolds, up to topological (orbital) equivalence.  

Following the pioneer works of Handel and Thurston \cite{HT} for the geodesic flow, Goodman \cite{Go} proved that for any Anosov flow $X$ on a manifold $M$ and any periodic orbit $\gamma$, one can build a new  Anosov flow on a manifold obtained from $M$ by a Dehn surgery along $\gamma$. In Goodman's construction, the dynamics of the new Anosov flow was not easy to understand. In \cite{Fri},  Fried proposed an alternative of the Dehn-Goodman surgery, for which the dynamics of the flow obtained from $X$ is topologically conjugated to $X$ except on $\gamma$. It was implicit in Fried's paper that his (topological) Anosov flow was indeed orbitally equivalent to the one obtained by Dehn-Goodman surgery and the math community generally admitted this during the 80's and 90's (see for instance \cite{Fe1}), before noticing that there was no explicit proof of such a statement. The orbital equivalence between Goodman's and Fried's surgery was indeed an open question. It was only recently that this was proven by Shannon, who proves that indeed Fried's surgery is orbitally equivalent to Dehn-Goodman surgery and that any topological Anosov flow is orbitally equivalent to an Anosov flow. \footnote{A first attempt to prove that a topological Anosov flow obtained by Fried's surgery is orbitally equivalent to a smooth Anosov flow was made by Marco Brunella in his thesis \cite{Bru}. However, Brunella's proof relied on the erroneous fact that isotopic pseudo-Anosov diffeomorphisms on surfaces with boundary are all conjugated.} 

Assume that $M$ is orientable and that $\gamma$ is a periodic orbit with positive eigenvalues. Then the boundary of a tubular neighbourhood of $\gamma$ is a torus endowed with a canonical \emph{meridian, parallel} basis of its fundamental group.  In this basis, the Dehn-Goodman-Fried surgery consists in keeping the same parallel and adding $n$ parallels to the meridian, we therefore speak of a \emph{surgery of characteristic number $n$}.  We also define a \emph{positive} or \emph{negative} surgery along $\gamma$ according to the sign of the characteristic number $n$.

One of the main open questions of this field (stated by Fried in \cite{Fri}) is 
\begin{ques}\label{c.Ghys} Any transitive Anosov flow is obtained through a finite sequence of Dehn-Goodman-Fried surgeries from the suspension flow of a hyperbolic linear automorphism of the torus $\TT^2$. 
\end{ques}

The aim of this paper is to study the Anosov flows obtained by a finite sequence of surgeries from a suspension Anosov flow, that is conjecturally, all the Anosov flows on $3$-manifolds. 

\subsection{$\RR$-covered and non-$\RR$-covered Anosov flows}\label{ss.RnonR}

Our point of view here is to consider the effect of surgeries on the \emph{bi-foliated plane} associated to an Anosov flow $X$, in order to decide if the flow in \emph{$\RR$-covered} or not. Let us recall these notions.

 \cite{Fe1} shows that for any Anosov flow $X$ on a $3$-manifold $M$, its lift $\tilde{X}$ on the universal cover $\tilde M$  is conjugated to the constant vector field $\frac\partial{\partial x}$ on $\RR^3$. The space of orbits of $\tilde X$ is therefore a $2$-plane $\cP_X\simeq \RR^2$, endowed with the natural quotient of the lift of the weak stable and unstable manifolds of $X$ on $\tilde M$. In other words, any Anosov flow $X$ is naturally associated to a pair of transverse foliations $F^s_X, F^u_X$ on the plane $\cP_X$: we call $(\cP_X,F^s_X,F^u_X)$ \emph{the bi-foliated plane associated to $X$}. 

In \cite{Fe1}, Fenley proves that if the space of leaves of $F^s_X$ is Hausdorff, then the same happens to the space of leaves of $F^u_X$. In this case, we say that $X$ is \emph{$\RR$-covered}.

If $X$ is $\RR$-covered, \cite {Fe1} shows that the bi-foliated plane $(\cP_X,F^s_X,F^u_X$) is conjugated to one of the two models:
\begin{itemize}
 \item $\RR^2$ endowed with the two foliations by parallel horizontal and vertical straight lines. We say in this case that $\RR^2$ is \emph{trivially bi-foliated}.
 \item the restrictions of the trivial horizontal/vertical foliations of $\RR^2$ to the strip $\{(x,y)\in\RR^2, |x-y|<1 \}$.  We say in this case that $X$ is \emph{twisted $\RR$-covered}.
\end{itemize}

\cite{Fe1} shows that if $X$ is an Anosov flow on an non-orientable manifold $M$, then it cannot be twisted $\RR$-covered: it is either trivially bi-foliated or non $\RR$-covered. 
 
From now on, the manifold $M$ is assumed to be oriented. In this case, the bi-foliated plane is naturally oriented.  If $X$ is twisted $\RR$-covered, the bi-foliated plane $(\cP_X,F^s_X,F^u_X$) is conjugated to one of the two models by an orientation preserving homeomorphism:
\begin{itemize}
 \item the restrictions of the trivial horizontal/vertical foliations of $\RR^2$ to the strip $\{(x,y)\in\RR^2, |x-y|<1 \}$. In this case, we say that
 $X$ is \emph{positively twisted $\RR$-covered}.
 \item the restrictions of the trivial horizontal/vertical foliations of $\RR^2$ to the strip $\{(x,y)\in\RR^2, |x+y|<1 \}$.  In this case, we say that
 $X$ is \emph{negatively twisted $\RR$-covered}.
\end{itemize}

For instance, the geodesic flow of a hyperbolic closed surface (or orbifold) is  $\RR$-covered. 
Here is an example which is typical of the results we obtain :

\begin{theo}\label{t.geodesic} Let $S$ be a hyperbolic closed surface and $X$ the geodesic flow on $M= T^1(S)$.  By choosing the orientation of $M$, we can assume $X$ to be $\RR$-covered positively twisted. 

Let $a_1,\dots ,a_k$ be a set of simple closed disjoint geodesics and $\Ga=\{\pm\gamma_i\}$ be the set of corresponding orbits of $X$. 

Then any flow $Y$ obtained from $X$ by surgeries on $\Ga$ is $\RR$-covered positively twisted. 
\end{theo}
In other words, surgeries on non intersecting closed geodesics have no effect on the bi-foliated plane.  

Non-$\RR$-covered flows also admit a specific set of orbits for which the corresponding surgeries have a very limited effect on the bi-foliated plane. More precisely, \cite{Fe2} defined the notion of pivot points in the bi-foliated plane $\cP_X$ of an non-$\RR$-covered Anosov flow and he proved that they correspond to a finite set ${\cP}iv(X)$ of periodic orbits.  We prove (see Theorems~\ref{t.pivot} and~\ref{t.cS} for more precise and stronger statements)
\begin{theo}\label{t.pivot0}
 Let $X$ be a non-$\RR$-covered Anosov flow and let $Y$ be obtained from $X$ by a finite number of Dehn surgeries along orbits of ${\cP}iv(X)$.  Then (up to the natural identification of the orbits of $Y$ with the orbits of $X$) one has ${\cP}iv(Y)={\cP}iv(X)$ and the stable (resp. unstable) leaves associated to two orbits $\gamma_{Y,1}$, $\gamma_{Y,2}$ of $Y$ admit representatives on $\cP_Y$ which are not separated if and only if the same happens for the correspondig orbits $\gamma_{X,1}$ and $\gamma_{X,2}$. 
\end{theo}
In \cite{Fe2}, Fenley proved that a non-$\RR$-covered Anosov flow has non-separated leaves in both $F^s_X$ and $F^u_X$ and that the non-separated leaves of $F^{s,u}_X$ correspond to finitely many periodic orbits. Theorem~\ref{t.cS} shows that surgeries on those periodic orbits also have a very limited effect on the bi-foliated plane. As a corollary of this fact one gets the following:
\begin{theo}\label{t.cS0} Let $X$ be a non-$\RR$-covered Anosov flow and let $Y$ be obtained from $X$ by a finite number of Dehn surgeries along periodic orbits whose stable leaves are not separated in $\tilde M$.  Then $Y$ is not $\RR$-covered and the set of periodic orbits whose unstable leaves are not separated coincides with the one of $X$.
\end{theo}

Our main results concern Anosov flows obtained by surgeries from a suspension. 

In this paper, $A\in SL(2,\ZZ)$ will denote a hyperbolic matrix (not necessarily of positive trace) and $f_A\colon \TT^2\to \TT^2$ the induced linear automorphism. We denote by $M_A,X_A$ the mapping torus manifold $M_A$ endowed with the suspension flow $X_A$. 
We will consider the set ${\cS}urg(A)$ of  Anosov flows obtained from $(M_A, X_A)$ through a finite sequence of Dehn-Goodman-Fried surgeries. A recent result, announced  by Minakawa \cite{Mi} and recently written by Dehornoy and Shannon \cite{DS} shows that if $A,B\in SL(2,\ZZ)$ are two hyperbolic matrices with positive eigenvalues then 
$$ {\cS}urg(A)={\cS}urg(B).$$
We will denote by ${\cS}urg_+$ this set. It is known that ${\cS}urg_+$ contains the geodesic flows of hyperbolic surfaces.

The aim of this paper is to describe the bi-foliated plane $(\cP_X, F^s_X,F^u_X)$ for $X\in {\cS}urg(A)$, as a  function of the surgeries (periodic orbits and characteristic numbers) performed on $X_A$ in order to obtain $X$. 

In \cite{Fe1}, Fenley shows that if $Y$ is an Anosov flow obtained from $X_A$ by performing finitely many Dehn surgeries, all positive, then $Y$ is twisted $\RR$-covered. 

Here we consider the effect of positive and negative surgeries on the bi-foliated plane.

\subsection{Statement of our main results for surgeries on an arbitrary set of periodic orbits}


Our most surprising result in this paper would certainly be the following: 

\begin{theo}\label{t.epsilon}  Let $A\in SL(2,\ZZ)$ be a hyperbolic matrix  and $X\in {\cS}urg(A)$. Then there is $\varepsilon>0$ such that for any periodic orbit $\gamma$ which is $\varepsilon$-dense all the flows $Y$ obtained from $X$ by surgeries along $\gamma$ are $\RR$-covered twisted positively or negatively according to the sign of the surgery on $\gamma$. 
\end{theo}

Conjecturally every Anosov flow on a oriented $3$-manifold belongs to ${\cS}urg_+$. It is therefore natural to ask if it is possible to prove: 
\begin{ques}\label{c.epsilon} Let $X$ be a transitive Anosov flow on a oriented $3$-manifold. Then there is $\varepsilon>0$ such that for any periodic orbit $\gamma$ which is $\varepsilon$-dense all the flows $Y$ obtained from $X$ by surgeries along $\gamma$ are $\RR$-covered twisted positively or negatively according to the sign of the surgery on $\gamma$. 
\end{ques}

Question~\ref{c.epsilon} is a straightforward consequence of Question~\ref{c.Ghys} and Theorem~\ref{t.epsilon}. Indeed, one can see Question~\ref{c.epsilon} as an intermediary step for proving Question~\ref{c.Ghys}. 

Contrary to Theorem \ref{t.epsilon} that describes a process of construction  of $\RR$-covered flows, our next result goes in the opposite direction, leading to the construction of non-$\RR$-covered Anosov flows. 

\begin{theo}\label{t.nonRcovered}Let $A\in SL(2,\ZZ)$ be a hyperbolic matrix  and $X\in {\cS}urg(A)$. Then there exist periodic orbits  $\gamma_+$ and $\gamma_-$ such that all the flows $Y$ obtained from $X$ by surgeries of distinct signs along $\gamma_+$ and $\gamma_-$  are not $\RR$-covered. 
\end{theo}

We proved indeed a much stronger result:

\begin{theo}\label{t.notRcoveredatall} Let $A\in SL(2,\ZZ)$ a hyperbolic matrix and $\cE$ be a finite $A$-invariant set.  Then there exist periodic orbits  $\gamma_+$ and $\gamma_-$ such that every flow $Y$ obtained from $X_A$ by any surgery on $\cE$ and any two surgeries of distinct signs along $\gamma_+$ and $\gamma_-$ is not $\RR$-covered.
\end{theo}
\textbf{Addendum to Theorem \ref{t.notRcoveredatall}.}
\textit{If $\cE$ is the union of the periodic orbits $p_1,...,p_n$ and we parametrise the surgeries performed on $\gamma_+ \cup \gamma_- \cup p_1 ...\cup p_n$ by $\ZZ^{n+2}$, there exist $D_+,D_-,P_1,...,P_n$ finite subsets of $\ZZ$ such that except $(D_+ \times \ZZ^{n+1}) \cup (\ZZ \times  D_-\times \ZZ^n)\cup (\ZZ^2 \times  P_1 \times \ZZ^{n-1})... \cup( \ZZ^{n+1} \times P_n)$ all the above surgeries yield non-$\RR$-covered flows $Y$ on hyperbolic manifolds.\vspace{0.2cm} }

Interestingly enough, Theorem~\ref{t.notRcoveredatall} implies that the surgeries on $\cE$ seem negligible in comparison with the ones on $\gamma_+$ and $\gamma_-$.  This is also the case for Theorem~\ref{t.epsilon} which also admits a stronger version: 

\begin{theo}\label{t.Epsilon} Let $\cE\subset \TT^2$ be a finite $f_A$-invariant set. There is $\varepsilon>0$ such that for any finite, $\varepsilon$-dense and $f_A$-invariant set $\cY\subset \TT^2$ one has the following: 

Let $Y$ be any flow obtained from $X_A$ by surgeries along $\cE\cup\cY$ and such that the characteristic numbers of the surgeries on $\cY$ are non zero and have the same signs $\omega_{\cY}\in\{+,-\}$.  Then $Y$ is $\RR$-covered and twisted, positively or negatively according to $\omega_{\cY}$. 
\end{theo}

\textbf{Addendum to Theorem \ref{t.Epsilon}.}
\textit{Furthermore, let $\cY$ (resp. $\cE$) be the union of the periodic orbits $d_1,...,d_n$ (resp. $p_1,...,p_m$). If we parametrize the surgeries performed on $d_1,...,d_n,p_1,...,p_m$ by $\ZZ^{n+m}$, then there exist $D_1,...,D_n, P_1,...,P_m$ finite subsets of $\ZZ$ such that except 
\begin{align*}
(&D_1  \times \ZZ^{n+m-1}) \cup (\ZZ \times  D_2 \times \ZZ^{n+m-2})...\cup(\ZZ^{n-1} \times D_n \times \ZZ^m)\\ &\cup (\ZZ^n \times P_1 \times \ZZ^{m-1}) \cup (\ZZ^{n+1} \times  P_2 \times \ZZ^{m-2})...\cup(\ZZ^{n+m-1} \times P_m )
\end{align*} all the above surgeries yield $\RR$-covered flows $Y$ on hyperbolic manifolds.\vspace{0.2cm}}

Thus Theorems~\ref{t.Epsilon} and~\ref{t.notRcoveredatall} consist in finding sets of periodic orbits $\cY$, on which well chosen surgeries dominate all surgeries on a given set $\cE$. Of course the sets $\cY \cup \cE$ are far from being generic. We are still very far from understanding the general case:

\begin{prob}\label{problem}
Consider a vector field $Y$ obtained from $X_A$ by performing positive surgeries on a finite $A$-invariant set $\Ga_+$, with strength $n_+\colon \Ga_+\to \NN^{*}$ and negative surgeries on a finite $A$-invariant set $\Ga_-$ with strength $n_-\colon \Ga_-\to -\NN^*$.   Knowing
$(\Ga_+,n_+),(\Ga_-,n_-)$, can we decide whether $Y$ is $\RR$-covered or not ?
 
\end{prob}

In the next section, we describe several settings where we can answer the previous question.
\subsection{Anosov flows obtained by surgeries from the suspension}\label{Anosov flows obtained from suspension by surgeries along $2$ orbits} 

Let us fix a hyperbolic matrix $A\in SL(2,\ZZ)$ (not necessarily of positive trace), $X_A$ its associated suspension Anosov flow and two disjoint finite $f_A$ invariant sets $\cX,\cY$. Consider $\cX$ (resp. $\cY$) to be the union of the periodic orbits $ \lbrace x_i\rbrace_{i\in I}$ (resp. $\lbrace y_j\rbrace_{j\in J}$). We denote by ${\cS}urg(X_A,\cX,\cY)$ the set of Anosov flows obtained by performing surgeries along $\cX \cup \cY$ and ${\cS}urg(X_A,\cX,\cY,(m_i)_{i\in I},\ast)$ the set of Anosov flows obtained by performing any kind of surgery along $\cY$ and surgeries with characteristic numbers $m_i$ along $x_i$. We give an analogous meaning to the notations ${\cS}urg(X_A,\cX,\cY,\ast, (n_j)_{j\in J})$ and ${\cS}urg(X_A,\cX,\cY,(m_i)_{i\in I}, (n_j)_{j\in J})$.

Let's first restrict to the simple case, where $\cX,\cY$ are respectfully equal to two distinct periodic orbits $\gamma_+$ et $\gamma_-$. There is a natural parametrization of ${\cS}urg(A,\gamma_+,\gamma_-)$ by the characteristic numbers of the surgeries along $\gamma_+$ and $\gamma_-$; therefore a parametrization by $\ZZ^2$.  

In this simple case, our goal is to describe in terms of $\gamma_+$ and $\gamma_-$, the regions of $\ZZ^2$ where we can decide whether ${\cS}urg(X_A,\gamma_+,\gamma_-,m,n)$ is $\RR$-covered or not, twisted (positively or negatively) or not.

Related to this problem is a question by Mario Shannon (that we do not answer here):
\begin{ques}[Shannon] Does there exist $A,\gamma_+,\gamma_-$, $\gamma_+\neq \gamma_-$ and $(m,n)\in\ZZ^2\setminus\{(0,0)\}$ such that ${\cS}urg(X_A,\gamma_+,\gamma_-,m,n)$ is a suspension flow? 
\end{ques}
Our results show that given $A,\gamma_+,\gamma_-$ there are at most finitely many $(m,n)$ for which the answer to the question can be positive. More generally, we think possible to prove that given a matrix $A$ there are at most finitely many $4$-uples $(\gamma_+,m,\gamma_-,n)$ for which the answer is positive. 

 According to \cite{Fe1} we know that:
\begin{itemize}
 \item if $m\geq 0$ and $n\geq 0$ and $(m,n)\neq (0,0)$ then ${\cS}urg(X_A,\gamma_+,\gamma_-,m,n)$ is $\RR$-covered and positively twisted. 
 \item if $m\leq 0$ and $n\leq 0$ and $(m,n)\neq (0,0)$ then ${\cS}urg(X_A,\gamma_+,\gamma_-,m,n)$ is $\RR$-covered and negatively twisted. 
\end{itemize}
When $m$ and $n$ have opposite signs, one could expect a competition between the effects of the surgeries along $\gamma_+$ and $\gamma_-$ on the bi-foliated plane, as they twist this plane in opposite directions: either one is dominating the other, leading to an $\RR$-covered twisted flow or the bi-foliated plane is positively twisted in some places and negatively in other places (whatever that means), leading to a non-$\RR$-covered flow.  We will see that the result of this competition depends on the mutual positions of the orbits $\gamma_+,\gamma_-$. Let us make this remark more precise and a bit more general.

We consider the plane $\RR^2$ (seen as the bi-foliated plane associated to $X_A$) endowed with the lattice $\ZZ^2$ and the oriented eigen directions $E^s_A, E^u_A$. One denotes by $F^s_A$ and $F^u_A$ the (trivial) foliations of $\RR^2$ by affine lines parallel to the eigen directions. We also denote by $\tilde{\cX},\tilde{\cY}$ the lifts of $\cX,\cY$ on $\RR^2$.  A \emph{rectangle} is a topological disc $R\subset \RR^2$, whose boundary consists of the union of two segments of leaves of $F^s_A$ and two segments of leaves of $F^u_A$. 

A rectangle $R$ has two \emph{diagonals}. The orientations of $E^s_A$ and $E^u_A$ allow us to speak of the \emph{increasing} and the \emph{decreasing} diagonal. We endow the diagonals with the transverse orientation of $E^s_A$, so that each diagonal has a \emph{first point} (or else \emph{origin}) and a \emph{last point}. 

If $\cE\subset \TT^2$ is a finite $f_A$-invariant subset of the torus $\TT^2= \RR^2/\ZZ^2$ we say that a rectangle $R$ is a 
\emph{positive} (resp. \emph{negative}) \emph{$\cE$-rectangle} if the endpoints of its increasing (resp. decreasing) diagonal belong to the lift $\tilde \cE$ on $\RR^2$ of $\cE$. 

A  positive or negative $\cE$-rectangle $R$ is \emph{primitive} if $R\cap\tilde \cE$ consists of the endpoints of its increasing or decreasing diagonal.  

Since $A$ is orientation preserving, $\cE$ is $f_A$-invariant and the foliations $F^s_A,F^u_A$ are invariant, one gets
\begin{rema}If $R$ is a rectangle, then $A(R)$ is a rectangle.  If $R$ is a positive $\cE$-rectangle, $A(R)$ is a positive $\cE$-rectangle. If $R$ is primitive, $A(R)$ is primitive. 
\end{rema}
In the same way, the notion of primitive (resp. positive, negative) $\cE$-rectangle is invariant under translations by elements of $\ZZ^2$. Also, as an immediate result of the previous remark we have:
\begin{rema}\label{r.periodicpoint}
If $\cE$ is the orbit of a periodic point and there exists a positive (resp. negative) $\cE$-rectangle with origin $x\in \tilde \cE$, then there exists a positive (resp. negative) $\cE$-rectangle with origin $y$ for all $y\in \tilde \cE$.
\end{rema}

\begin{lemm}\label{l.orbitesfinies} For any finite $f_A$ invariant set $\cE \subset \TT^2$, there are finitely many orbits of primitive $\cE$-rectangles, for the action of the group generated by $A$ and the integer translations.  
\end{lemm}

This lemma makes the hypotheses of Theorem~\ref{t.twisted} reasonable:

\begin{theo}\label{t.twisted}  Let $A\in SL(2,\ZZ)$ be a hyperbolic matrix and $\cX,\cY$ two disjoint finite $f_A$ invariant sets. 
Assume that every positive $\cX$-rectangle contains a point of $\tilde{\cY}$. 

Then there is $N>0$ such that every Anosov flow in ${\cS}urg(X_A,\cX,\cY,\ast, (n_j)_{j\in J})$ with $n_j\geq N$ is $\RR$-covered and positively twisted. 
\end{theo}

Obviously, the same statement holds:
\begin{itemize}\item  by exchanging $\cX$ with $\cY$
 \item by replacing positive rectangle and positively twisted by negative rectangle and negatively twisted. 
\end{itemize}

Using the Lemma\ref{l.orbitesfinies}, we get that many pairs $(\cX,\cY)$ satisfying the hypotheses of Theorem~\ref{t.twisted}:

\begin{lemm}\label{l.epsilon} Given any $f_A$-invariant finite set $\cX$, there is $\varepsilon>0$ such that every $\varepsilon$-dense finite invariant set $\cY$ intersects every $\cX$-rectangle. Such a pair $(\cX,\cY)$ satisfies the hypotheses of Theorem~\ref{t.twisted}. 
\end{lemm}
Finally, as a corollary of Theorem~\ref{t.twisted} we get the following: 
\begin{lemm}\label{l.asymetric} If $\tilde{\cY}$ intersects every positive $\cX$-rectangle, then there is a negative $\cY$-rectangle disjoint from $\tilde{\cX}$. 
 
\end{lemm}

\begin{theo}\label{t.string}
Let $A\in SL(2,\RR)$ be a hyperbolic matrix and $\cX,\cY$ two disjoint finite $f_A$ invariant sets. Assume that for every $x\in \cX$ there exists a positive $\cX$-rectangle with origin $x$ disjoint from $ \tilde{\cY}$ and for every $y\in \cY$ a negative $\cY$-rectangle with origin $y$ disjoint from $ \tilde{\cX}$.

Then there exists $N>0$ such that every Anosov flow of the form ${\cS}urg(X_A,\cX,\cY,(m_i)_{i\in I}, (n_j)_{j\in J})$ with $m_i\leq -N$ and $n_j\geq N$ is not $\RR$-covered. 
\end{theo}
Once again the same statement holds by straightforward symmetries. 
\vspace{5mm}
\par{In the ``simple" case when $\cX,\cY$ are periodic orbits, the previous theorems and Remark \ref{r.periodicpoint} allow us to get an almost complete answer to the problem \ref{problem}. One first considers $16$ cases corresponding to the existence (or not) of positive (negative) $\cX$ (resp. $\cY$) rectangles disjoint from $\tilde \cY$ (resp $\tilde \cX$). Lemma \ref{l.asymetric} allows us to discard seven of the above cases.}

Up to exchanging $\cY$ with $\cX$, up to changing the orientation of $\RR^2$ and up to all other possible symmetries, we restrict the possible cases to the following ones: 

\begin{enumerate}
 \item There are positive and negative $\cX$-rectangles disjoint from $\tilde{\cY}$ and positive and negative $\cY$-rectangles disjoint from $\tilde{\cX}$.  In this case, there exists $N>0$ such that ${\cS}urg(X_A,\cX,\cY,m, n)$ is not $\RR$-covered when $m>N$ and $n <-N$ or $m<-N$ and $n>N$. It is $\RR$-covered when $n$ and $m$ have the same sign and it is twisted according to the sign of $m$ or $n$. 
 
 \item All positive $\cX$-rectangles intersect $\tilde{\cY}$ . There are negative $\cX$-rectangles disjoint from $\tilde{\cY}$ and positive and negative $\cY$-rectangles disjoint from $\tilde{\cX}$. 
 
 ${\cS}urg(X_A,\cX,\cY,m, n)$ is $\RR$-covered twisted when $m<0$ and $n>N$ where $N$ is positive and large enough or when $m$ and $n$ have the same sign. It is not $\RR$-covered when $m>0$, $n<0$ and both $|m|,|n|$ are large enough. 
 
 \item All positive and negative $\cX$-rectangles intersect $\tilde{\cY}$.
${\cS}urg(X_A,\cX,\cY,m, n)$ is $\RR$-covered when $mn>0$ or when $|n|$ is large enough. The direction of the twist in this case follows the sign of the $n$.

 \item All positive $\cX$-rectangles intersect $\tilde{\cY}$ and all positive $\cY$-rectangles intersect $\tilde{\cX}$. 
 
 ${\cS}urg(X_A,\cX,\cY,m, n)$ is $\RR$-covered when $mn>0$ or when either $n$ or $m$ is positive with large absolute value. 
 
\end{enumerate}

\subsection{Structure of the paper}
In Section~\ref{s.prelim} we remind basic definitions and properties of Anosov flows on $3$-manifolds. In particular, we remind Fenley's works on the bi-foliated plane, a characterization of $\RR$-covered and non $\RR$-covered Anosov flows and finally some properties of the Dehn-Goodman-Fried surgery. 

In Section~\ref{s.chirurgies} we remind  very basic facts allowing us to compare the bi-foliated planes associated to two Anosov flows $X$ and $Y$ obtained one from the other by surgeries. This leads to a general procedure defined in Theorem~\ref{t.newholonomies} for comparing the holonomies of the foliations of both bi-foliated planes. When $X$ is a suspension flow, the procedure in Theorem~\ref{t.newholonomies} can be made more explicit and will be called \emph{the dynamical game for computing the holonomies}.  

In Section~\ref{s.geodesic} we give a general criterion (see Corollary~\ref{c.caractR}), ensuring that surgeries on a finite set of periodic orbits cannot break the $\RR$-covered property. Then we apply Corollary~\ref{c.caractR} to the geodesic flow of hyperbolic surfaces and we prove Theorem~\ref{t.geodesic}. 

In Section~\ref{s.nonR} we prove Theorems~\ref{t.pivot} and~\ref{t.cS}, which are more precise and stronger versions of  Theorems~\ref{t.pivot0} and~\ref{t.cS0} concerning surgeries which do not change the branching structure of non-$\RR$-covered Anosov flows. It consists essentially in recalling the description of this branching structure given in \cite{Fe2} and in applying the general tools of Section~\ref{s.chirurgies}. 

From this point on, we focus on Anosov flows $X\in \cS urg_+$, i.e. Anosov flows obtained from a suspension flow by surgeries. 

More particularly, Section~\ref{s.rectangle} ends with the proof of  Theorems~\ref{t.epsilon} and Theorems~\ref{t.Epsilon} in which we prove that for $X\in \cS urg_+$ any surgery on an $\varepsilon$-dense periodic orbit, for $\varepsilon>0$ small enough, provides an $\RR$-covered flow. In order to prove the previous statement, we begin by proving Theorem~\ref{t.twisted} and we proceed by carefully replacing the \emph{strong enough surgeries} hypothesis in Theorem~\ref{t.twisted} by the $\varepsilon$-density hypothesis.

Section~\ref{s.string}, being the non-separated counterpart of Section~\ref{s.rectangle}, follows the same structure. We begin by proving Theorem~\ref{t.string} and we proceed by reclacing the \emph{strong enough surgeries} condition by an $\varepsilon$-density condition, thus proving Theorems~\ref{t.nonRcovered} and~\ref{t.notRcoveredatall}.  

In Section~\ref{s.deux} we consider the flows $X\in\cS urg_+$ obtained from a suspension by surgeries on $2$ periodic orbits.  In this case, by applying Theorems~\ref{t.twisted} or \ref{t.string} we get a complete overview of the flows $X$ obtained from $X_A$ by strong enough surgeries. 

Sections~\ref{s.rectangle} and~\ref{s.string} consider mostly surgeries on orbits of very large period (the period of $\varepsilon$-dense orbit tends to infinity as $\varepsilon$ goes to $0$).  In order to present explicit examples of low periods ($1$ or $3$), Section~\ref{s.explicit} focusses on the matrices 
$$A_n=\left(\begin{array}{cc}
             n&n-1\\
             1&1
            \end{array}\right)$$
and their cubes $B_n=A_n^3$. We will apply the criteria of Sections~\ref{s.rectangle} and~\ref{s.string} to the orbits of the points $(0,0)$ and $(\frac{1}{2},\frac12)$.

\subsubsection*{Acknowledgements} We would like to address a special thanks to Sergio Fenley for his interest and comments and also to Fran\c ois B\'eguin for organizing the \emph{groupe de travail sur les flot d'Anosov} -(by videoconference) during the 2020 quarantine that allowed us to present very early versions of the results.  This work is part of the \emph{stage de recherche de quatri\`eme ann\'ee} (4th year research training program at the ENS de Lyon) of the second author. 

\section{$\RR$-covered and non-$\RR$-covered Anosov flows on $3$-manifolds}\label{s.prelim}

\subsection{Anosov flows:  definitions, stability, orbital equivalence}\label{s.basic}

\begin{defi}\label{d.Anosov} A $C^1$-vector field $X$ on a closed manifold $M$ is called \emph{an Anosov flow} if the tangent bundle $TM$ admits a splitting 
$$ TM= E^s\oplus \RR X\oplus E^u$$
satisfying the following properties:
\begin{itemize}
 \item the splitting is invariant under the natural action of the derivative $DX^t$ of the flow on $TM$: 
 $$ DX^t(E^s(x))=E^s(X^t(x)) \mbox{ and } DX^t(E^s(x))=E^s(X^t(x)).$$
 \item if $\|\cdot\|$ is a Riemannian metric on $M$, there is $C>0$ and $0<\lambda<1$ such that, for any $x\in M$ and any two vectors $u\in E^s(x)$ and $v\in E^u(x)$ one has
 $$\|DX^t(u)\|\leq C \lambda^{t}\|u\| \mbox{ and } \|DX^{-t}(v)\|\leq C \lambda^{t}\|v\|$$
 
\end{itemize}
\end{defi}

An important property of Anosov flows is 

\begin{theo}\cite{A} If $X$ is an Anosov flow, then there is $C^1$-neighborhood $\cU$ of $X$ such that every $Y \in \cU$ is topologically (orbitally) equivalent to $X$: there is a homeomorphism $h\colon M\to M$ such that for every $x\in M$  the image of the oriented orbit of $X$ through $x$ is the oriented orbit of $Y$ through $h(x)$.

 One says that $X$ is \emph{$C^1$-structurally stable}.
\end{theo}
The homeomorphism $h$ in the theorem can be chosen isotopic to the identity map. 

We denote by $\cA(M)$ the set of orbital equivalence classes of Anosov flows and by  $\cA_0(M)$ the set of equivalence classes of Anosov flows through orbital equivalence by homeomorphisms isotopic to the identity.
Theorem~\cite{A} implies that the set $\cA_0(M)$ is at most countable on any closed manifold $M$. The set $\cA(M)$, being a quotient of $\cA_0(M)$, is at most countable too.

There are simple examples of manifolds $M$ for which $\cA_0(M)$ is infinite (consider for instance the image of the geodesic flow of a hyperbolic surface by a vertical diffeomorphism of the unit tangent bundle).  Up to now it was unknown if there are manifolds for which $\cA(M)$ is infinite. There are $3$-manifolds for which $\cA(M)$ has a cardinal greater than any given number, see \cite{BeBoYu}.\footnote{An example of manifold $M$ for which $\cA(M)$ is infinite has been recently proposed in \cite{CP}. }

\subsection{Foliations}

Another important property of Anosov flows is that the stable, center-stable unstable and center unstable fiber bundles 
$E^s,E^{cs}=E^s\oplus\RR X, E^u, E^{cu}=E^u\oplus\RR X$ are uniquely integrable.

\begin{theo} There are unique foliations $\cF^s,\cF^{cs}, \cF^u, \cF^{cu}$ tangent to $E^s,E^{cs}, E^u, E^{cu}$.  More precisely any $C^1$ curve tangent to one of these bundles is contained in a leaf of the corresponding foliation.  
These foliations are invariant under the flow of $X$. 
\end{theo}
The foliations $\cF^s,\cF^{cs}, \cF^u, \cF^{cu}$ are respectively called \emph{stable, center-stable, unstable, center unstable} . 

In dimension $3$ the center stable and center unstable foliations provide the main known obstructions for a $3$-manifold $M$ to carry an Anosov flow. 

\begin{theo} A leaf $L$ of the center stable (or center unstable foliation) is
\begin{itemize}
 \item diffeomorphic to a plane $\RR^2$ if and only if $L$ does not contain a periodic orbit. 
 \item diffeomorphic to a cylinder $\RR\times S^1$ if it contains a periodic orbit of $X$ with positive stable eigenvalue; the periodic orbit in $L$ is unique. 
 \item diffeomorphic to a M\oe bius band if it contains a periodic orbit with negative stable eigenvalue; the periodic orbit in $L$ is unique. 
\end{itemize}
\end{theo}

As a direct corollary of the above, the manifold $M$ carries foliations ($\cF^{cs}$ and $\cF^{cu}$) with no compact leaves and thus with no \emph{Reeb component}. Under these hypotheses, a consequence of Novikov's theorem implies that $M$ admits $\RR^3$ as a universal cover.

A simple argument allows also to check that the leaves have an exponential growth. As a consequence of this, the fundamental group of $M$ has exponential growth (see \cite{PlTh}).

\subsection{The bi-foliated plane associated to an Anosov flow on a $3$-manifold, $\RR$-covered and non $\RR$-covered Anosov flows. } 
Before beginning this section, the reader can refer to Section~\ref{ss.RnonR} for the definitions of
\begin{itemize}\item the bi-foliated plane $(\cP_X, F^s_X,F^u_X)$ associated to an Anosov flow $X$ on a $3$-manifold.
\item a non-$\RR$-covered Anosov flow
\item an $\RR$-covered Anosov flow
\item a twisted $\RR$-covered Anosov flow
\item a positively and negatively twisted $\RR$-covered Anosov flow (this notions are only defined on oriented manifolds and depend on a choice of the orientation )
\end{itemize}

The fundamental group $\pi_1(M)$ acts by the deck transformation group on the universal cover $\tilde M\simeq \RR^3$ of $M$. This action preserves the lift $\tilde X$ of $X$ on $\tilde M$ and also  preserves the lifts $\tilde\cF^{cs}_X,\tilde\cF^{cu}_X$ of the center stable and center unstable foliations. Therefore, the action passes to the quotient. By quotienting we obtain an action of $\pi_1(M)$ on $\cP_X$, which preserves the foliations $F^s_X$ and $F^u_X$. This action is called \emph{the natural action} of $\pi_1(M)$ on the bi-foliated plane $(\cP_X, F^s_X, F^u_X)$ and we denote it 
$$\theta_X\colon \pi_1(M)\to Homeo(\cP_X, F^s_X, F^u_X),$$ 
where $Homeo(\cP_X, F^s_X, F^u_X)$ is the group of homeomorphisms of the plane $\cP_X$ preserving the foliations $F^s_X$ and $F^u_X$. 

As we consider only Anosov flows on oriented manifolds, the action $\theta_X$ takes values in $Homeo_+(\cP_X, F^s_X, F^u_X)$. 

If $X$ is a transitive Anosov flow then the action on $\cP_X$ admits dense orbits.  Furthermore, the orbit of any half leaf of $F^s_X$ and of $F^u_X$ is dense in $\cP_X$.

Consider an element $\tilde \gamma\in\cP_X$  corresponding to a periodic orbit $\gamma\subset M$.  Then one has a well defined  element  $[\tilde \gamma]\in\pi_1(M)$  which is the homotopy class of a closed path obtained by the concatenation $\sigma\gamma\sigma^{-1}$ where $\sigma$ is a path joining the base-point in $M$ to a point $x$ of $\gamma$ and which is the projection of a path joining the base point in $\tilde M$ to a point of the orbit associated to $\tilde \gamma$.  

The following lemma is a classical result in the theory (see for instance \cite{Barbot}).
\begin{lemm}\label{l.centralizer} Let $\tilde \gamma\in\cP_X$ be a point corresponding to a periodic orbit $\gamma$ of $X$. We also denote by $\tilde\gamma$ the corresponding orbit in $\tilde M$.  Consider $G_{\tilde\gamma}\subset \pi_1(M)$ its stabilizer for the natural action of $\theta$. 

Then $G_{\tilde \gamma}$ is the cyclic group  generated by  the homotopy class  $[\tilde \gamma]$ .  
\end{lemm}

Let $\gamma$ be a periodic point. We say that a curve  $L^s\subset W^s(\gamma)$ is a \emph{complete transversal} if it transverse to $X$ such that the first return map of $X$ induces a homeomorphism $P_\gamma\colon L^s\to L^s$ (which is a contraction).  
One defines in the same way a \emph{complete transversal $L^u \subset W^u(\gamma)$} and the first return map $P_\gamma\colon L^u\to L^u$ which is a dilation. 

Now using the previous notations, $L^s$ and $L^u$ admit canonical lifts on $\tilde M$ through the point $\tilde \gamma$ and these lifts project on $\cP_X$ injectively. Let $h^s$ and $h^u$ be the projections from $L^s$ and $L^u$ on $F_X^s(\gamma)$ and $F_X^u(\gamma)$.   
We denote by $P_{\tilde \gamma}\colon F^s(\tilde \gamma)\to F^s(\tilde\gamma)$ and $P_{\tilde \gamma}\colon F^u(\tilde \gamma)\to F^u(\tilde \gamma)$ the homeomorphisms $h^s P_{\gamma} (h^s)^{-1}$ and $h^u P_{\gamma} (h^u)^{-1}$. We can easily be convinced that

\begin{lemm}\label{l.firstreturn} The homeomorphism $P_{\tilde\gamma}\colon F^s_{\tilde\gamma}\cup F^u_{\tilde \gamma}\to F^s_{\tilde\gamma}\cup F^u_{\tilde \gamma}$ is independent of the choices of $\tilde \gamma$, of $L^s$ and $L^u$, and is called the first return map of $X$.  
\end{lemm}

The first return map is related to the homotopy class $[\tilde \gamma]\in\pi_1(M)$ as follows:

\begin{lemm} The natural action  $\theta_{X,[\tilde\gamma]}$ of $[\tilde\gamma]$ on $\cP_X$ preserves $F^s(\tilde \gamma)$ and $F^u(\tilde \gamma)$ and its restriction to $F^s(\tilde \gamma)\cup F^u(\tilde \gamma)$ is $P_{\tilde\gamma}^{-1}$. 
\end{lemm}
A proof of the previous lemma can be found in \cite{Barbot}.

\subsection{A charaterization of $\RR$-covered Anosov flows by complete and incomplete quadrants}

Let $(\cF,\cG)$ be two oriented transverse foliations on the plane 
$\cP=\RR^2$. This defines  four quadrants at each point $x$:  
for any $\omega=(\omega_1,\omega_2)\in \{-,+\}^2$, the (closed) quadrant $C_\omega(x)$ is the closure of the connected component of 
$\cP\setminus (\cF(x)\cup \cG(x))$ bounded by the half leaves 
$\cF_{\omega_1}(x),\cG_{\omega_2}(x)$.

\begin{defi}
 \begin{itemize}
 \item  We say that the pair $(\cF, \cG)$ is \emph{undertwisted} or \emph{incomplete}  in the quadrant $C_{(+,+)}(x)$ if there are $y\in \cF^+(x)$ and $z\in\cG^+(x)$ such that 
 $$\cG^+(y)\cap\cF^+(z)=\emptyset.$$
 \item We say that the pair $(\cF, \cG)$ is \emph{complete} (or has \emph{the complete intersection property}) in the quadrant $C_{(+,+)}(x)$ if for all  $y\in \cF^+(x)$ and $z\in\cG^+(x)$  one has  
 $$\cG^+(y)\cap\cF^+(z)\neq \emptyset.$$
 
The complete case is divided in two subcases
 \item  $(\cF,\cG)$ is \emph{trivial} in the quadrant $C_{+,+}(x)$ if 
 $$ \bigcup_{y\in \cF_+(x)}\cG_+(y)=\bigcup_{z\in \cG_+(x)}\cF_+(z). $$
 
 \item The pair $(\cF,\cG)$  is \emph{overtwisted} in the quadrant $C_{+,+}(x)$  if it is complete but not trivial. In other words 
for all  $y\in \cF^+(x)$ and $z\in\cG^+(x)$ 
one has  $\cG^+(y)\cap\cF^+(z)\neq \emptyset$ but there is $p\in C_{+,+}(x)$ such that $\cF(x)\cap \cG(p)=\emptyset$ or $\cG(x)\cap \cF(p)=\emptyset$. 
 \end{itemize}
 
 One defines these notions in all the other quadrants in the same way, changing some $+$ into $-$ according to the quadrant. 
\end{defi}

\begin{rema}\label{r.Rcoveredquadrants}
 \begin{itemize}
  \item If $X$ is a suspension, every quadrant is trivial
  \item If $X$ is $\RR$-covered and positively twisted, then every quadrant $C_{+,+}(x)$, $C_{-,-}(x)$ are complete and overtwisted, and the quadrants $C_{+,-}(x)$ $C_{-,+}(x)$ are understwisted. 
  \item If $X$ is $\RR$ covered and negatively twisted, then the quadrants $C_{+,-}(x)$, $C_{-,+}(x)$ are complete and overtwisted and the quadrants $C_{+,+}(x)$ $C_{-,-}(x)$ are understwisted.
 \end{itemize}
\end{rema}

\begin{lemm}Let $(\cF,\cG)$ be a pair of oriented transverse foliations and assume that two leaves $L_1$ and $L_2$ in $\cF$ are not separated from above, i.e. there are two positively oriented $\cG$-leaf segments $\sigma_i\colon [0,1]\to \cP$ such that $\sigma_i(0)\in L_i$ and $\sigma_1(t)$ and $\sigma_2(t)$ belong to the same $\cF$-leaf for all $t>0$. 

Then there exist $x,y\in\cP$ such that $C_{-,-}(x)$ and $C_{+,-}(y)$ are incomplete (undertwisted). 
\end{lemm}
\begin{proof} $x$ and $y$ are points $\sigma_i(t),\sigma_j(t)$, $t>0$. 
\end{proof}

\begin{lemm}A transitive Anosov flow $X$ is a suspension if and only if there is 
$x\in\cP_X$ and $\omega\in\{+,-\}^2$ such that the quadrants $C_\omega(x)$ and $C_{-\omega}(x)$ are trivially foliated. 
\end{lemm}
\begin{proof} An Anosov suspension flow has clearly trivially foliated quadrants; we only need to prove the converse. Assume that $X$ is a transitive Anosov flow, whose bi-foliated plane has a trivial quadrant, say $C_{+,+}(x)$. Note that $C_{+,+}(y)$ is trivial too for any $y\in C_{+,+}(x)$.  As $X$ is transitive, there is $y\in C_{+,+}(x)$ with a dense orbit. One deduces that for any $z\in \cP_X$, there is $\gamma\in \pi_1(M)$ such that $z\in C_{+,+}(\theta_\gamma(y))$.  Thus 
$C_{+,+}(z)$ is trivial for any $z\in \cP(x)$. Of course by the same exact method, we obtain that $C_{-,-}(z)$ is also trivial for any $z\in \cP(x)$.

As a consequence of this, $F^s_X$ and $F^u_X$ don't contain non-separated leaves, hence $X$ is $\RR$-covered. Finally, by Remark \ref{r.Rcoveredquadrants} it cannot be twisted, so the bi-foliated plane is trivial and $X$ is a suspension. 
\end{proof}



 
 

We are now ready to state our criteria for deciding whether an Anosov flow is $\RR$-covered or not: 

\begin{coro}An Anosov flow $X$ is not $\RR$-covered if and only if there are $x,y\in\RR^2$  and 
$\omega_x,\omega_y\in \{-,+\}^2$ which are \emph{adjacent} (i.e distinct and not opposite) such that the quadrants $C_{\omega_x}(x)$ and $C_{\omega_y}(y)$ are incomplete. 
\end{coro}

\begin{coro}\label{c.caractR} Let $X$ be an Anosov flow and assume that for every $x,y\in\RR^2$ the quadrants 
$C_{+,+}(x)$ and $C_{-,-}(y)$ are complete.  Then 
either the bi-foliation is trivial or the flow is $\RR$-covered and positively twisted. 
\end{coro}



\begin{rema}Given a transitive Anosov flow $X$, the set of $x\in \cP_X$ such that the pair $(F^s_X,F^u_X)$ is incomplete in $C_{(+,+)}(x)$ is either empty or is a dense open subset. 
\end{rema}

\subsection{Stable and unstable holonomies and the completeness of the quadrants}


Fix $x\in \cP_X$ and consider $y\in F^u_{X}(x)$. We call \emph{unstable holonomy from $x$ to $y$} the map $h^u_{X,x,y}$ from $F^s_X(x)$ to $F^s(y)$  defined by 

$$\{h^u_{X,x,y}(z)\}= F^u_{X}(z)\cap F^s_X(y),\mbox{ for } z\in F^s_X(x), \mbox{ if } F^u_{X}(z)\cap F^s_X(y)\neq \emptyset$$
This definition is consistent as the intersection of a stable and an unstable leaves is at most one point.

The domain $\cD(h^u_{X,x,y})$ is an interval of $F^s_X(x)$ and the image is an interval of $F^s_X(y)$. 

One defines in a analogous way \emph{the stable holonomy $h^s_{X,x,z}$} for $z\in F_X^s(x)$.

\begin{rema} For $x\in\cP_X$ the quadrant $C_{+,+}(x)$ is complete if for any $y\in F^u_+(x)$ one has 
$$F^s_+(x)\subset \cD(h^u_{X,x,y})$$

The quadrant $C_{+,+}(x)$ is undertwisted if there is $y\in F^u_+(x)$ such that $\cD(h^u_{X,x,y})$ is a bounded interval in $F^s_+(x)$.

Analogous statements hold in every quadrant and by exchanging the unstable holonomy by the stable one. 
\end{rema}

\subsection{Dehn-Goodman-Fried surgery}
As explained in the introduction, it has been proved recently that the topological flow built by Fried's surgery is orbitally equivalent to the Anosov flow obtained by Goodman's surgery. Thanks to its explicitness, throughout the next pages we will rather make use of the action of  Fried's surgery on the bi-foliated plane instead of its general definition that we quickly remind now. 

Let $X$ be an Anosov flow on a oriented $3$-manifold $M$ and let $\gamma$ be a periodic orbit with positive eigenvalues. 

Consider the \emph{blow-up $\pi_\gamma\colon M_\gamma\to M$} of $M$ along $\gamma$, that is:
\begin{itemize}
\item $M_\gamma$ is a manifold with boundary, and $\partial M_\gamma$ is a torus $T_\gamma \simeq \TT^2$.
\item $\pi_\gamma$ induces a diffeomorphism from the interior of $M_\gamma$ to $M\setminus\gamma$. 
\item for every $x\in\gamma$ the fiber $\pi_\gamma^{-1}$ is a circle which is canonically identified with the unit normal bundle $N^1(x)$ of $\gamma$ in $M$ at the point $x$.

In other words, consider two segments $\sigma_1,\sigma_2$
in $M_\gamma$ transverse to the boundary $\partial M_\gamma$ at $\sigma_i(0)$.  Then $\sigma_1(0)=\sigma_2(0)$ if and only if 
$\pi_{\gamma}(\sigma_1(0))=\pi_{\gamma}(\sigma_2(0))=c$ and the segments $c_1=\pi_\gamma\circ \sigma_1$ and $c_2=\pi_\gamma\circ \sigma_2$ have the following property: the vector $\frac{\partial c_2}{\partial t}(0)$ belongs to the half plane of $T_{c(0)}M$ containing $\pm X(c(0))$ and $\frac{\partial c_1}{\partial t}(0)$. 

\end{itemize}
The vector field $\pi_\gamma^{-1}(X)$ is  well defined  on the interior of $M_\gamma$,  extends by continuity on the boundary $T_\gamma$ by the natural action of the derivative $DX^t$ on the normal bundle over $\gamma$. We denote by $X_\gamma$ this (smooth) vector field on $M_\gamma$.

The flow on $T_\gamma$ is a Morse-Smale flow with $4$ periodic orbits, which correspond to the normal vectors to $\gamma$ tangent to the stable and unstable manifolds of $\gamma$.  These $4$ periodic orbits are freely homotopic one to 	another and are non trivial in $\pi_1(T_\gamma)$.  The homotopy (or homology) class $b\in\ZZ^2=\pi_1(T_\gamma)$ of these periodic orbits is called \emph{the parallel} . 
 
On the other hand the fibers of $\pi_\gamma\colon T_\gamma\to \gamma$ inherit an orientation from the orientation of $M$, and the corresponding homotopy class $a\in\ZZ^2=\pi_1(T_\gamma)$ is called \emph{the meridian} . 

Given any integer $n\in\ZZ$, one easily checks the existence of foliations $\cG_n$  on $T_\gamma$, transverse to the flow $X_\gamma$, whose leaves are simple closed curves of homotopy class $a+nb$. By reparametrizing  the flow $X_\gamma$, one gets a new smooth vector field  $Y_\gamma$ on  $M_\gamma$  that leaves invariant the foliation $\cG_n$.

Let $M_{\gamma,n}$ be the manifold obtained from $M_\gamma$ by collapsing the leaves of $\cG_n$. The flow $Y_\gamma$ passes to the quotient and becomes a topological Anosov flow $X_{\gamma,n}$ on $M_{\gamma,n}$. 

It is easy to do this construction in a way such that $X_{\gamma,n}$ is a Lipschitz vector field, but it is not clear at all that it can be smooth and Anosov. It is not even clear that the orbital equivalence class of the construction does not depend on the choice of the foliation $\cG_n$. 
Shannon proved that it is orbitally equivalent (by a homeomorphism isotopic to identity) to an Anosov flow (the one built by Goodman), proving at the same time that the orbital equivalence class of this construction (in fact, the element of $\cA_0(M_{\gamma,n})$) is well defined.  This element of $\cA_0(M_{\gamma,n})$ is called \emph{the Anosov flow $X_{\gamma,n}$ obtained from $X$ by a surgery along $\gamma$ with characteristic number $n$}. 

\begin{rema}
\begin{itemize}
 \item If $\gamma_1,\gamma_2$ are periodic orbits of $X$ and $n_1,n_2$ 
 are integers, then 
 $$[X_{\gamma_1,n_1}]_{\gamma_2,n_2}=[X_{\gamma_2,n_2}]_{\gamma_1,n_1}.$$
 In other words, the surgeries are commutative operations.  This allows us to speak without any ambiguity of the  Anosov vector field $Y$ obtained from $X$ by performing surgeries on periodic orbits $\gamma_1,\dots,\gamma_k$ with characteristic numbers $n_1,\dots,n_k$. 
 \item  If $\gamma$ is a periodic orbit and $m,n\in \ZZ$ then
 $$[X_{\gamma,m}]_n=X_{\gamma,m+n}.$$
\end{itemize}

\end{rema}

\begin{rema} 
\begin{itemize}\item A similar construction holds for periodic orbits $\gamma$ with negative eigenvalues, but the parallel and meridian intersect twice and thus are not a basis of $\pi_1(T_\gamma)$. 
\item This construction is local and thus can be done on non-orientable manifolds for orientation preserving periodic orbits. However, the characteristic number depends on the local orientation and thus is not well defined for non-orientable manifolds. 
\end{itemize}
\end{rema}
\subsubsection{Fried surgeries leading to hyperbolic manifolds}
\par{In this article, we're interested in constructing Anosov flows on hyperbolic manifolds. There are two reasons behind this: first Anosov flows on hyperbolic manifolds satisfy additional structural properties -for instance any periodic orbit of any $\RR$-covered Anosov flow on a hyperbolic manifold is freely homotopic to infinitely many periodic orbits (see \cite{Fe1})- and second, to our knowledge, there are no examples of non-$\RR$-covered Anosov flows on hyperbolic manifolds. }

In this paper, we provide a construction of infinitely many $\RR$-covered and non-$\RR$-covered Anosov flows on hyperbolic manifolds. The most important step of the hyperbolic part of the construction, which also proves the addenda of Theorems \ref{t.notRcoveredatall} and \ref{t.Epsilon} is the following lemma:
\begin{lemm}\label{l.hypmanifold}
Let $X$ be the suspension flow of an Anosov diffeomorphism of the torus $\TT^2$ and $M$ its underlying manifold. Fix $\gamma_1,...\gamma_n$ periodic orbits of $X$. There exist $D_1,..D_n$ such that for any $(k_1,...,k_n) \in \ZZ^n- [(D_1\times \ZZ^{n-1}) \cup (\ZZ \times D_2 \times \ZZ^{n-2}) ... \cup (\ZZ^{n-1} \times D_n)]$, $[[[X_{\gamma_1,k_1}]_{\gamma_2,k_2}]...]_{\gamma_n,k_n}$ is an Anosov flow on a hyperbolic manifold.

\end{lemm}
\begin{proof}
In \cite{Th} Thurston showed that $M -\bigcup_{i=1}^n\gamma_i$ admits a complete hyperbolic structure of finite volume and thanks to the hyperbolic Dehn surgery theorem (see \cite{Th1}) we obtain the desired result. 
\end{proof}

\section{The surgeries and the bifoliated plane}\label{s.chirurgies}

Let $X$ and $Y$ be two Anosov flows on closed $3$-manifolds $M$ and $N$ such that (the orbital equivalence class of) $Y$ is obtained from $X$ by performing finitely many surgeries along peridoc orbits:   there exist $k\in\NN$,  a finite set $\Ga=\{\gamma_1,\dots,\gamma_k\}$ of periodic orbits of $X$,  and   a finite set  $\cN=\{n_1,\dots,n_k\}\subset \ZZ$ such that $Y= X_{\Ga,\cN}$, that is $Y$ is the topological Anosov flow obtained from $X$ by performing a Fried surgery with caracterisctic number $n_i$ on each $\gamma_i$.

The aim of this section is to give a very partial answer to the following question: 
\begin{ques}
Knowing the bi-foliated plane $(\cP_X,F^s_X,F^u_X)$ what can we say about \\$(\cP_Y, F^s_Y,F^u_Y)$ ?
\end{ques}

A key remark for answering this question is that $\Ga$ may be included in $N$ and 
$$M\setminus \Ga = N\setminus \Ga \mbox{ and } X|_{M\setminus \Ga}= Y|_{N\setminus \Ga}.$$

\subsection{The key tool: a common cover}

We'll denote by $\tilde \Ga_X\subset \tilde M$ and $\tilde \Ga_Y\subset \tilde N$ the lifts of  $\Ga$ to the universal covers $\tilde M$ and $\tilde N$. 

By a convenient abuse of language, we'll still denote by $\tilde \Ga_X$ and $\tilde \Ga_Y$ the corresponding (discrete) sets in $\cP_X$ and $\cP_Y$. 

Let $V=M\setminus \Ga=N\setminus \Ga$ and $Z$ be the restriction of $X$ to $V$ (or equivalently of $Y$ to $V$). 
\begin{clai}The universal cover $(\tilde V,\tilde Z)$ is conjugated to $(\RR^3,\partial/\partial x)$,
\end{clai}
\begin{proof} $\tilde V$ is the universal cover of $\tilde M\setminus \tilde \Ga_X$ which is conjugated to $\RR^3$ minus a discrete family of orbits of $\partial/\partial x$ which are parallel straight lines. 
\end{proof}

The space of orbits in $\tilde V$ is a bi-foliated plane, denoted by $(\cP_{\Ga}, F^s_\Ga, F^u_\Ga)$.  This bi-foliated plane is the universal cover of $(\cP_X,F^s_X,F^u_X)\setminus \tilde \Ga_X$ and of 
$(\cP_Y,F^s_Y,F^u_Y)\setminus \tilde \Ga_Y$.  We denote by $\Pi_X$ and $\Pi_Y$ the natural projections of $\tilde V$ onto $\cP_X$ and $\cP_Y$.

$$
\begin{array}{rcccl}
&  &\cP_\Ga & & \\
&\overset{\Pi_X\;\;\;}\swarrow& &\overset{\;\;\;\Pi_Y}\searrow& \\
\cP_X\setminus \tilde \Ga_X& & & &\cP_Y\setminus\tilde\Ga_Y
\end{array}
$$

This simple fact has an important (straighforward) consequence:
\begin{lemm} Let $R_X\subset \cP_X$ be a rectangle for $(F^s_X,F^u_X)$ disjoint from $\tilde \Ga_X$ and $R_\Ga$ a connected component of $\Pi_X^{-1}(R_X)$. 

Then $R_Y= \Pi_Y(R_\Ga)$ is a rectangle for $F^s_Y,F^u_Y$ and $\Pi_Y\circ \Pi_X^{-1}$ induces a homeomorphism from $R_X$ to $R_Y$ conjugating $(F^s_X,F^u_X)$ and $(F^s_Y,F^u_Y)$.  
 
\end{lemm}
\begin{proof} $R_\Ga$ is a rectangle and $\Pi_X$ takes the bi-foliated $R_\Ga$ to the bi-foliated $R_X$. The only thing to check now is that $\Pi_Y$ restricted on $R_\Ga$ is a homeomorphism. It suffices to prove that $\Pi_Y$ is injective on the rectangle $R_\Ga$. If $\Pi_Y(x)=\Pi_Y(y)$ for $x\neq y\in R_\Ga$, then the stable and unstable leaves at $\Pi_Y(x)$ intersect twice which is impossible in a (non singular) bi-foliated plane. 
\end{proof}

One can use the following generalization of this argument: 

\begin{prop}\label{p.rectangle-closure}
Consider a closed domain $\De_X\subset \cP_X$ such that its interior is disjoint from $\tilde \Ga_X$ and $(\De_X, F^s_X|_{\De_X},F^u_X|_{\De_X})$ is conjugated to the trivially bi-foliated plane $\RR^2$.  

Let $\De_\Ga$ be a connected component of $\Pi_X^{-1}(\De_X)$  and let  $\De_Y$ be the closure of $\Pi_Y(\De_\Ga)$.  

Then  $\Pi_Y\circ \Pi_X^{-1}$ (defined on $\De_X\setminus \tilde \Ga_X$) extends on $\De_X$ to a homeomorphism conjugating $(F^s_X,F^u_X)$ to $(F^s_Y,F^u_Y)$.  
\end{prop}




\subsection{Two ways for associating a path in $\cP_X$ to a path in $\cP_Y$  \label{ss.2ways}}
Let $\sigma_X\colon \RR \to \cP_X$ be a locally injective continuous path, obtained by the concatenation of locally finitely many stable / unstable leaf segments. One can define a transverse orientation as follows : the transverse orientation followed by the orientation of $\sigma_X$ is the orientation of $\cP_X$.
\begin{rema}\label{r.orientation} For this choice of the orientation, 
\begin{itemize}\item the transverse orientation of a positively oriented unstable segment coincides with the orientation of the stable leaves intersecting it.
 \item the transverse orientation of a positively oriented stable segment coincides with the negative orientation of the unstable leaves intersecting it. 
\end{itemize}
\end{rema}

Assume that $p_X= \sigma_X(0) \notin \tilde \Ga_X$ and $p_Y\in \Pi_Y(\Pi_X^{-1}(p_X))$. 

Let $\sigma_{X,t}\colon \RR \to \cP_X$, $t\in[-1,1)$ be a continuous family of paths such that: 
\begin{itemize}
 \item  $\sigma_{X,0}=\sigma_X$
 \item $\sigma_{X,t}(0)=p_X$
 \item $\sigma_{X,t}$ is disjoint from $\tilde\Ga_X$
 \item $\sigma_{X,t}$ tends to $\sigma_{X,0}$ from the positive side as $t\to 0$ and  $t>0$ and from the negative side as $t\to 0$ and $t<0$. 
\end{itemize}

\begin{coro}\label{c.corresponding-pathes} Using the above notations, there are uniquely defined paths $\sigma_{Y,t}$ for $t>0$ (resp. $t<0$) such that 
\begin{itemize}
 \item $\sigma_{Y,t}(0)=p_Y$
 \item $\sigma_{Y,t}(s)\in\Pi_Y(\Pi_X^{-1}(\sigma_{X,t}(s)))$, $s\in\RR$
\end{itemize}
 
 and the limits
 $$\lim_{t\to 0_+}\sigma_{Y,t}(s)=\sigma_{Y,+}(s) \mbox{ and }\lim_{t\to 0_-}\sigma_{Y,t}(s)=\sigma_{Y,-}(s)$$
are well defined continuous paths which are a concatenation of locally finitely many stable/unstable segments. Furthermore, $\sigma_{Y,+}$ and $\sigma_{Y,-}$ only depend on the choice of $p_Y$ and not on the choice of the homotopies $\sigma_{X,t}$.  

\end{coro}
\begin{proof} We construct $\sigma_{Y,t}$ by lifting $\sigma_{X,t}$ on $\cP_\Ga$, which is the universal cover of $\cP_X\setminus \tilde\Ga_X$, and then projecting the lifts by $\Pi_Y$ on $\cP_Y\setminus \tilde\Ga_Y$. 

The second part of the statement is a consequence of Proposition~\ref{p.rectangle-closure}. 
\end{proof}
\begin{defi}\label{d.correponding-pathes} Using the above notations, $\sigma_{Y,+}$ and $\sigma_{Y,-}$ are called \emph{the positive} and \emph{negative} (respectively) \emph{paths through $p_Y$ corresponding to $\sigma_X$. We can similarly define this notion in the case where $p_X\in \tilde\Ga_X$.} 
 
\end{defi}

\begin{figure}[h!]
\includegraphics[scale=0.60]{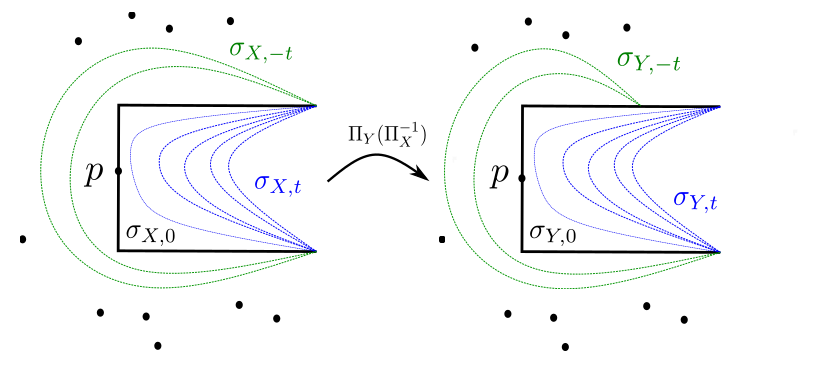}
\caption{In above picture the black points represent points in $\tilde \Ga_{X,Y}$ on which we've performed non-trivial surgeries}
\label{f.holonomy}
\end{figure}
It is fairly easy to see that, in general, $\sigma_{Y,+}$ and $\sigma_{Y,-}$ do not coincide. An example of such a case is given in figure \ref{f.holonomy}. In this example, we consider a black path, a family of green and blue paths all with the same endpoints in $\cP_X$. Because of the surgery performed on $p$, by applying $\Pi_Y \circ \Pi_X^{-1}$ to a blue and a green path, we obtain two paths in $\cP_Y$ that do not share the same endpoints. The previous is going to be made more precise in Proposition \ref{p.franchissement}.

\subsection{Comparison of holonomies: the main tool}

From this point on, unless explicitly said otherwise, if $\Phi$ is an Anosov flow acting on $M$, we orient $M$ in the following way: we arbitrarily choose an orientation of its stable and unstable directions and we consider the orientation on $M$ given by the stable direction, followed by the unstable direction, followed by the flow direction. 

Our main tool for comparing the holonomies of the foliations associated to $X$ and $Y$ is the next proposition:

\begin{prop}\label{p.franchissement} Let $\tilde \gamma \in\cP_X$ be a point corresponding to a periodic orbit $\gamma\in\Ga$.  Let $n$ be the characteristic number of the surgery performed on $\gamma$.

Consider two rectangles $R^+,R^-$ such that:
\begin{itemize}
\item $R^{\pm}\cap\tilde \Ga_X=\{\tilde \gamma\}$. 
 \item $\tilde \gamma$ belongs to the interior of the lower stable boundary component $J^s$ of $R^+$ (see figure \ref{f.rectanglesholonomy}) and to the interior of the upper stable boundary component $I^s$ of $R^-$. 

 \item The positively oriented stable segments $I^s,J^s$ satisfy $I^s=[a,b]^s$ and $J^s=[a,c]^s$, with $c= P_{\tilde\gamma}^{n}(b)$, where $P_{\tilde\gamma}$ is the first return map on $F^s(\tilde \gamma)$ and $F^u(\tilde\gamma)$.

 \end{itemize}
 If we denote $R^{\pm}=R^+ \cup R^-$, then for any connected component $R^{\pm}_{\Ga}$ of $\Pi_X^{-1}(R^\pm)$ the projection $\Pi_Y(R^{\pm}_\Ga)$ is a rectangle punctured at $\tilde \gamma$. 
 
\end{prop}
\begin{figure}[h!]
\includegraphics[scale=0.75]{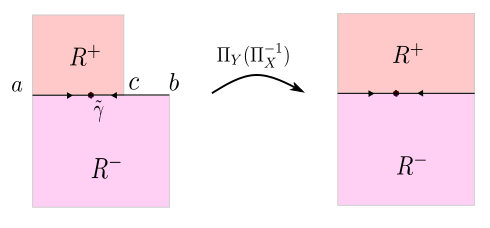}
\caption{}
\label{f.rectanglesholonomy}
\end{figure}
\begin{proof} 

Consider the closed curve $\delta$ on $\cP_X$ starting at the point $c$ following $\partial R^+\setminus Int (J^s)$ until $a$ then following $\partial R^-\setminus Int (I^s)$ until $b$.  Project $\delta$ on a local section of $\gamma$ and complete it by the orbit segment joining $b$ to $c$.  Then the obtained closed curve is freely homotopic, in $M\setminus \gamma$ to a meridian plus $n$ parallels of $\gamma$ that is to the new meridian after surgery. 

Thus $\delta$ is $0$-homotopic on the manifold $N$ carrying $Y$. Its lifts on $\cP_Y$ are closed curves consisting of $2$ stable and $2$ unstable segments, hence bounding a rectangle, which finishes the proof. 

\end{proof}

In Proposition~\ref{p.franchissement} one considers an orbit segment joining the points $b,c \in F^s_+(\tilde \gamma) $ by turning around $\tilde \gamma$ in the positive sense.  Analogous statements hold after changing the sign of the exponent of the first return map. Let us be more explicit, as this is crucial for our arguments. 
 
\begin{prop}\label{p.franchissement2}
Let $\tilde \gamma \in\cP_X$ be a point corresponding to a periodic orbit $\gamma\in\Ga$.  Let $n$ be the characteristic number of the surgery performed on $\gamma$.

Consider two rectangles $R^+,R^-$ such that:
\begin{itemize}
\item $R^{\pm}\cap\tilde \Ga_X=\{\tilde \gamma\}$. 
 \item $\tilde \gamma$ belongs to the interior of the lower stable boundary component $J^s$ of $R^+$ and to the interior of the upper stable boundary component $I^s$ of $R^-$. 

 \item The positively oriented stable segments $I^s,J^s$ satisfy $I^s=[b,a]^s$ and $J^s=[c,a]^s$, with $c= P_{\tilde\gamma}^{-n}(b)$, where $P_{\tilde\gamma}$ is the first return map on $F^s(\tilde \gamma)$ and $F^u(\tilde\gamma)$.

 \end{itemize}
 If we denote $R^{\pm}=R^+\cup R^-$, then for any connected component  $R^{\pm}_\Ga$ of $\Pi_\Ga^{-1}(R^\pm)$ the projection $\Pi_Y(R^{\pm}_\Ga)$ is a rectangle punctured at $\tilde \gamma$. 
 
\end{prop}
\begin{proof}The proof is identical to the one of Proposition~\ref{p.franchissement}, except that in this case the path from $c$ to $b$ is negatively oriented: one obtains $-1$ meridian minus $n$ parallels for $X$, which is $-1$ meridian for $Y$. 
 
\end{proof}

Finally, let us note that the previous results hold independently of the sign of the eigenvalues of $\gamma$. 
\subsection{Comparison of the holonomies in the quadrants: choosing the holonomies to be compared}

Our next goal in this paper is to obtain the holonomies of $F^s_Y$ and $F^u_Y$ from the holonomies of $F^s_X$ and $F^u_X$, by using Proposition~\ref{p.franchissement}. In Section~\ref{ss.formula}, we will first describe the change of holonomies for the unstable holonomies in the $C_{+,+}$ quadrants and then we will explain how to adapt the previous statement in all the other quadrants.
 
Before doing that, we need to explain which holonomies of $F^u_Y$ and $F^u_X$ are going to be compared. 
 
 More precisely, consider a point $p_X\in \cP_X$, a point $q_X\in F^u_+(p_X)$ and the unstable holonomy $h^u_{X,p_X,q_X}$ from $F^s_{X,+}(p_X)$ to $F^s_{X,+}(q_X)$. We want to describe the effect of the surgery on this holonomy and to compare the new holonomy with $h^u_{X,p_X,q_X}$.  
 
 Consider the path $\sigma_X$ obtained by the concatenation of:
 \begin{itemize}\item the half stable leaf $F^s_{X,+}(p_X)$ (with the negative orientation) 
  \item the unstable segment $[p_X,q_Y]^u$
  \item the half stable leaf $F^s_{X,+}(q_X)$.
 \end{itemize}
 
 We fix a parametrization of $\sigma_X$ so that the path $\sigma_X$ becomes a map $\sigma_X\colon \RR\to \cP_X$. 
 
 \subsubsection{The easy case: no point of $\tilde \Ga$ on $\sigma_X$}
 Assume first that $\sigma_X$ is disjoint from $\tilde \Ga_X$. Consider a lift $p_\Ga\in\cP_{\Ga}$ of $p_X$ and let $p_Y\in \cP_Y$ be the projection of $p_{\Ga}$. 
 
 Now $\sigma_X$ has a well defined lift $\sigma_{\Ga}$ on $\cP_{\Ga}$ through $p_\Ga$ 
and consider $\sigma_Y$ the projection of $\sigma_{\Ga}$. 

In other words, we have  
$$\sigma_Y=\Pi_Y(\Pi^{-1}_X(\sigma_X))$$ 
and $\sigma_Y(t)=\Pi_Y(\Pi^{-1}_X(\sigma_X))$ will be called \emph{the corresponding point of $\sigma_X(t)$ in $\cP_Y$}.  
 
Thus $q_Y$ is the corresponding point of $q_X$ in $\cP_Y$ and it belongs to $F^u_+(p_Y)$. So the unstable holonomy $h^u_{Y,p_Y,q_Y}$ from $F^s_{Y,+}(p_Y)$ to $F^s_{Y,+}(q_Y)$ is well defined.

As every point in $F^s_+(p_X)$ (resp. in $F^s_+(q_X)$) has a corresponding point in $F^s_+(p_Y)$ (resp. in $F^s_+(q_Y)$), it makes sense to compare  $h^u_{X,p_X,q_X}$ with $h^u_{Y,p_Y,q_Y}$. 

\subsubsection{The general case} 

In the next section, we will need to compare holonomies of $X$ and $Y$ corresponding to stable leaves that intersect $\tilde \Ga_X$ and $\tilde \Ga_Y$, in other words, we will consider the case where $\sigma_X$ is not disjoint from $\tilde\Ga_X$. In this case, $\sigma_X$ no longer lifts on $\cP_\Ga$. We have seen in Section~\ref{ss.2ways} that one may associate different paths $\gamma_Y$ in $\cP_Y$ to a path $\gamma$ in $\cP_X$, depending, roughly speaking, on whether we choose to move at the right or at the left of $x$ for every $x$ in $\gamma_X\cap \tilde \Ga_X$. 

The aim of this subsection is to fix our choices for the segment $\sigma_X$.

Consider a point $p_X\in \cP_X$ and $p_Y\in \cP_Y$, which are obtained in one of the following ways:
\begin{itemize}\item either as the projections of a same point $p_{\Ga} \in \cP_{\Ga}$; this means in particular that $p_X\notin \tilde \Ga_X$. 
 \item or, if $p_X\in \tilde \Ga_X$, we consider a small rectangle $R_{X,+,+}$ admitting $p_X$ as its lower left corner and such that $R_{X,+,+} \cap \tilde\Ga_X=\{p_X\}$. We lift $R_{X,+,+}\setminus \{p_X\}$ on $\cP_\Ga$ and we project this lift on $\cP_Y$. One gets a rectangle 
 $R_{Y,+,+}$  punctured at its lower left corner, which is denoted by $p_Y\in \tilde \Ga_Y$. 
\end{itemize}

Given a point $q_X$ in $F^u_{X,+}(p_X)$, we defined previously a path
$\sigma_X$  obtained by the concatenation of $\sigma^1_{X}=F^s_{X,+}(p_X)$ (with the negative ortientation),  $\sigma^2_{X}=[p_X,q_Y]^u$ and $\sigma^3_X=F^s_{X,+}(q_X)$. 

The point $q_Y\in \cP_Y$ corresponding to $q_X$ will be the end point of the path $\sigma^2_{Y,+}$, defined in Section~\ref{ss.2ways} and whose origin is $p_Y$. 

We want to compare the unstable holonomy $h^u_{X,p_X,q_X}$ with the unstable holonomy  $h^u_{Y,p_Y,q_Y}$ for $Y$. 
 
In order to do that, we consider the path $\sigma_Y$ obtained by the concatenation of the three paths:
\begin{itemize}
\item the half stable leaf $\sigma^1_{Y,+}$ (which corresponds to  $F^s_{Y,+}(p_Y)$, negatively oriented)
\item the unstable segment  $\sigma^2_{Y,+}$ (joining $p_Y$ to $q_Y$)
\item the half stable leaf $\sigma^3_{Y,-}$ (which corresponds to $F^s_{Y,+}(q_Y)$, positively oriented)
\end{itemize}

The construction of the paths $\sigma^i_{Y,\pm}$ (see Section~\ref{ss.2ways})  induces a homeomorphism from $\sigma^i_X$ to $\sigma^i_Y$ (maping $\sigma^i_X(t)$ on $\sigma^i_Y(t)$.  By glueing these homeomorphisms together, we get a homeomorphism between $\sigma_X$ and $\sigma_Y$, mapping $\sigma_X(t)$ on $\sigma_Y(t)$.

 


For any point $z_X=\sigma_X(t)$, we define its \emph{corresponding point} as $z_Y:=\sigma_{Y}(t)$. 

\begin{rema}Our choice to define $\sigma_Y$ as the concatenation of the paths $\sigma^1_{Y,+}$  $\sigma^2_{Y,+}$  and $\sigma^3_{Y,-}$ may look surprising. It corresponds to the projection on $\cP_Y$ of the lifts of a sequence of (continuous) paths $\sigma_n$ disjoint from $\tilde \Ga_X$ and approching $\sigma_X$ from above on $F^s_{X,+}(p_X)$, from the right on $[p_X,q_X]^u$ and from above on $F^s_{X,+}(q_X)$.  In particular the paths $\sigma_n$ are contained in $C_{+,+}(p_X)$ and cross $F^s_{X,+}(q_X)$.

 In terms of the holonomy $h^u_{Y,p_Y,q_Y}$, this means that 
 we consider that the segments $[y,h^u_{Y,p_Y,q_Y}(y)]^u$ do not cross $F^s_{Y,+}(p_Y)$ but cross $F^s_{Y,+}(q_Y)$. (Another choice would not change the holonomy  $h^u_{Y,p_Y,q_Y}$ itself, but would change the parametrization of the path $\sigma_Y$). 
 
 This choice is convenient for composing holonomies. 
\end{rema}

\subsection{Comparison of the holonomies in the quadrants: the formula\label{ss.formula}}
 
We are now ready to compare the holonomies $h^u_{X,p_X,q_X}$ and $h^u_{Y,p_Y,q_Y}$. 

\begin{theo}\label{t.newholonomies} With the notations above, let $x_Y\in F^s_{Y,+} (p_Y)$. 
Then 
$$h^u_{Y,p_Y,q_Y}(x_Y)=y_Y$$
if and only if there exist $\ell\in\NN$ and two finite sequences 
$t_i\in \RR$ and $x_i\in\cP_X$ with $i\in\{0,\dots, \ell\}$ such that 

\begin{enumerate}
 \item $x_0=x_X$ and $x_l=y_X$ 
 \item $\sigma_X(t_0)=p_X$, $\sigma_X(t_\ell)=q_X$
 \item $t_i<t_{i+1}$ for $i\in\{0,\dots \ell-2\}$ and $t_{\ell-1} \leq t_{\ell}$, 
 therefore $\sigma_X(t_i)\in [p_X,q_X]^u$. We denote $q_{X,i}=\sigma_X(t_i)$,
 \item for $i\in\{1,\dots,\ell-1\}$ there exists $\mu_i\in\tilde \Ga_X$ (see figure \ref{f.rectanglesholonomy}) such that the point $q_{X,i}$  belongs to $F^s_{X,-}(\mu_i)$ and the point $x_i$ belongs to $F^s_{X,+}(\mu_i)$. We denote by $k_i$ the corresponding characteristic number of the surgery and we take $k_0=0$.  
 \item $\{x_1\}=F_X^u(x_{0})\cap F_X^s(q_{X,1})$ and $\{x_{i+1}\}=F_X^u(P^{k_i}_{\mu_i}(x_{i}))\cap F_X^s(q_{X,i+1})$ (where $P_{\mu_i}$ is the first return map of $X$ associated to $\mu_i$ see Lemma~\ref{l.firstreturn}) for $i \in\{1,\dots,\ell-1\}$. 
  
  \item  let $R_i$ for $i\in \{0,\dots,\ell-1\}$ be the rectangle ($R_{\ell-1}$ can be degenerated) bounded by the segments 
 $[q_{X,i},q_{X,i+1}]^u$, 
 $[q_{X,i+1}, x_{i+1}]^s$, $[q_{X,i}, P^{k_i}_{\mu_i}(x_i)]^s$ and 
 $[P^{k_i}_{\mu_i}(x_i),x_{i+1}]^u$ . 
 Then the interior of $R_i$ is disjoint from $\tilde \Ga_X$.

\end{enumerate}

\end{theo}
\begin{proof}
If $q_X$ doesn't belong to the negative stable manifold of a point $\mu$ on which we've performed surgery, the above theorem is obtained by a simple induction argument using Proposition \ref{p.franchissement}. 

Otherwise, we can use a simple induction argument to calculate the holonomy from $F^s_{X,+}(p_X)$ to $F^s_{X,+}(q^-_X)$, where $q^-_X \in [p_X,q_X]^u$ and satisfies the hypothesis of the previous case. $q^-_X$ can be taken as close as we want to $q_X$.  By Proposition \ref{p.franchissement} and because of our choice of $\sigma_X$ and $\sigma_Y$, we have that changing the surgery on $\mu$ would change the parametrization of the path $\sigma''_{Y,-}$. Therefore, in order to compute the holonomy from $F^s_{X,+}(q^-_X)$ to $F^s_{X,+}(q_X)$ we must  apply Proposition \ref{p.franchissement} for two rectangles $R^-$ and $R^+$, where $R^+$ is degenerated. 
\end{proof}
\begin{figure}[h!]
\includegraphics[scale=0.85]{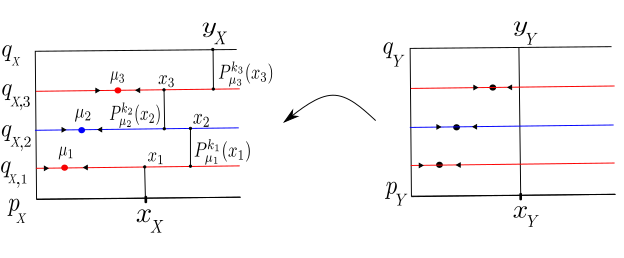}
\caption{In the above figure we performed negative surgeries along the red periodic points and positive along the blue ones. Every time we hit a stable manifold of either a blue or red point the holonomy is respectfully contracted or expanded.}
\label{f.rectanglesholonomy}
\end{figure}
Let us make some remarks about the previous theorem:
\begin{rema}In the above theorem, $\sigma_X$ and $\sigma_Y$ play the roles of local coordinates on each bi-foliated plane. Changing the definition of the above coordinates would naturally change the statement of the theorem and therefore the computation of the holonomy. 
\end{rema}
\begin{rema}
\begin{itemize}\item \quad The same statement holds for the holonomies in the $C_{-,-}$ quadrant by changing $F^s_{X,+}$ and $F^u_{X,+}$ to $F^s_{X_-}$ and $F^u_{X,-}$. In fact, it is enough to apply Theorem~\ref{t.newholonomies} after changing the orientation of both foliations $F^s_X$ and $F^u_X$. This change preserves the orientation of the manifold and hence preserves the characteristic numbers $k_i$ of the surgery. 
\item An analogous statement holds in the $C_{+,-}$ $C_{-,+}$ quadrants, but one needs to change the sign of the characteristic numbers, therefore to change $k_i$ to $-k_i$. 

For that, we apply Theorem~\ref{t.newholonomies} after changing only one of the two orientations of $F^s_X$ and $F^u_X$, thus changing the orientation of $M$ and eventually the orientation of the meridian.  For this new orientation, the characteristic number of the surgery changed sign.
\item Once again Theorem \ref{t.newholonomies} holds independently of the sign of the eigenvalues of the periodic points in $\tilde \Ga_X$

\end{itemize}
\end{rema}

\subsection{The special case where $X$ is a suspension: calculating the holonomies as a dynamical game}\label{ss.game}

In this section, we assume that $X$ is the suspension flow of a hyperbolic Anosov diffeomorphism $f_A$, where $A\in SL(2,\ZZ)$ is a hyperbolic matrix with positive eigenvalues $0<\lambda^{-1}<1<\lambda$ (taking negative eigenvalues changes nothing to our arguments).

In this particular case, the bi-foliated plane is trivial, so the holonomies are also trivial in $\cP_X$ and the first return maps are simple to understand.  This will simplify significantly the statement of Theorem~\ref{t.newholonomies}. 

We perform a linear change of coordinates on $\cP_X=\RR^2$ so that $F^s_X$ is the horizontal foliation and $F^u_X$ is the vertical folation. In these coordinates,  $A$ is the linear map 
$$\cA=\left(\begin{array}{cc}
             \lambda^{-1}&0\\
             0&\lambda
            \end{array}\right)$$

We consider now:
\begin{itemize}\item two finite sets $\cX$ and $\cY$ of $\TT^2$ which are disjoint and $f_A$ invariant. Every point $x$ in $\cX\cup\cY$ is periodic and we denote by $\tau(x)$ its period. Notice that $\tau$ is invariant by $f_A$. 
 \item two functions $m\colon \cX\to\NN$ and $n\colon \cY\to \NN$ which are 
 $f_A$-invariant. 

\end{itemize}

By a convenient abuse of language, we'll still denote by $\cX$ and $\cY$ the periodic orbits of the vector field $X$. In this way, the functions $m$ and $n$ become  integer functions on this finite set of orbits of $X$.

We denote by $\tilde \cX$ and $\tilde \cY$ the lifts of $\cX$ and $\cY$ on $\cP_X$, which is canonically identified with the universal cover of the torus $\TT^2$.  We still denote by $\tau$, $m$ and $n$ the lifts of the previous functions. 

In the previous section, we defined the first return map $P_{ x}$ associated to a point $x\in\cP_X$ corresponding to a periodic orbit of $X$. In our setting, the first return map associated to a point $ x\in\tilde\cX\cup\tilde\cY$ is the affine map having $ x$ as its unique fixed point and $\cA^{\tau(x)}$ as its linear part. 

$$P_{x}= p\mapsto \cA^{\tau(x)}(p-x) + x$$
            
We denote by $Y$ the vector field obtained from $X$ by performing surgeries with characteristic numbers $m$ on the orbits in $\cX$ and $-n$ on the orbits in $\cY$.

We denote 
$$\mu= m\cdot \tau \mbox{ and } \nu=n\cdot \tau.$$

The aim of this section is to express Theorem~\ref{t.newholonomies} in this particular setting. 

Consider a point $p=p_X=(p^s,p^u)\in\RR^2$. We want to describe the holonomies of $F^s_Y$ and $F^u_Y$ in the quadrants $C_{\pm,\pm}(p_Y)$, where $p_Y$ is the projection on $\cP_Y$ of a lift of $p_X$ on the universal cover of $\cP_X\setminus (\tilde \cX\cup\tilde \cY)$.  Let us start from the $C_{+,+}$ quadrant, in order to avoid useless formalism. 

\subsubsection{On the $C_{+,+}$ quadrants}

Consider $r>0$, $t_0>0$, a point $q=(p^s, p^u+r)$ in the positive unstable manifold of $p$, a point $z_0=(p^s+t_0, p^u)$ in the positive stable manifold of $p$ and $p_Y,z_{0,Y}$ their corresponding points in $\cP_Y$.  One would like to know if the holonomy $h^u_{p,q}$ of $F^u_Y$ 
 from the positive stable manifold of $p_Y$ to the positive stable manifold of $q_Y$ is defined on $z_{0,Y}$ and if this is the case, what is its value. 
 
In order to answer the previous question, one considers the set of points $z_s=(p^s+t_0, p^u+s)$ in $\cP_X$, with $0<s < s_1 \leq r$, where $s_1$ is the smallest positive real for which there exists a point $\gamma_1=(p^s+u_1,p^u+s_1)\in\tilde \cX\cup\tilde\cY$, with $0<u_1<t_0$. If such an $s_1$ doesn't exist, then the holonomy from $W^s_+(p)$ to $W^s_+(q)$ is defined on $(p^s+t_0, p^u)$ and its value is $(p^s+t_0, p^u+r)$

If such an $s_1$ exists (see figure \ref{f.jeuholonomiesuspension}), then one defines $z_s=(p^s+t_1, p^u+s)$, with $s\in [s_1,s_2)$ where 
\begin{itemize}
\item $t_1= u_1+\lambda^{-\mu(\gamma_1)}(t_0-u_1)$ if $\gamma_1\in\tilde \cX$ and 
\item $t_1= u_1+\lambda^{\nu(\gamma_1)}(t_0-u_1)$ if $\gamma_1\in\tilde \cY$  
\item $s_2$ is the smallest positive number in $(s_1,r]$  such that $z_t$ crosses the positive stable manifold of a point $\gamma_2=(p^s+u_2,p^u+s_2)\in\tilde \cX\cup\tilde\cY$, with $0<u_2<t_1$. 
\end{itemize}

\begin{figure}[h!]
\includegraphics[scale=0.85]{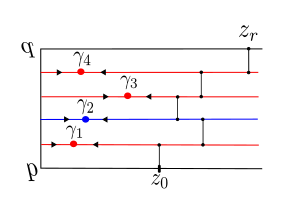}
\caption{In the above figure the periodic points on which we performed positive surgery are represented by blue and the others by red. }
\label{f.jeuholonomiesuspension}
\end{figure}
 Analogously, if such an $s_2$ doesn't exist then then the holonomy from $W^s_+(p)$ to $W^s_+(q)$ is defined on $(p^s+t_0, p^u)$ and its value is $(p^s+t_1, p^u+r)$. If $s_2$ exists, then one defines $z_s=(p^s+t_2, p^u+s)$, $s\in [s_2,s_3)$ where 
\begin{itemize}
\item $t_2= u_2+\lambda^{-\mu(\gamma_2)}(t_1-u_2)$ if $\gamma_2\in\tilde \cX$ and 
\item $t_2= u_2+\lambda^{\nu(\gamma_2)}(t_1-u_2)$ if $\gamma_2\in\tilde \cY$  
\item $s_3$ is the smallest positive number in $(s_2,r]$  such that $z_t$ crosses the positive stable manifold of a point $\gamma_3=(p^s+u_3,p^u+s_3)\in\tilde \cX\cup\tilde\cY$, with $0<u_3<t_2$.
\end{itemize}

We define by induction the sequences $t_i$, $s_{i+1}$, $u_{i+1}$, $\gamma_{i+1}$:  
 $z_s=(p^s+t_i, p^u+s)$, $s\in [s_i,s_{i+1})$ where 
\begin{itemize}
\item $t_i= u_i+\lambda^{-\mu(\gamma_i)}(u_i-t_{i-1})$ if $\gamma_i\in\tilde \cX$ and 
\item $t_i= u_i+\lambda^{\nu(\gamma_i)}(u_i-t_{i-1})$ if $\gamma_i\in\tilde \cY$  
\item $s_{i+1}$ is the smallest positive number in $(s_i,r]$ such that $z_t$ crosses the positive stable manifold of a point $\gamma_{i+1}=(p^s+u_{i+1},p^u+s_{i+1})\in\tilde \cX\cup\tilde\cY$, with $0<u_{i+1}<t_i$.
\end{itemize}

Now, 
\begin{itemize}
 \item either this process is repeated infinitely many times. In this case, the holonomy $h^u_{p,q}$ is not defined at the point $z_0$. 
 \item or the process ends when for some $i\in \mathbb{N}$, $s_i$ is not defined. In this case, $h^u_{p,q}$ is defined at the point $z_0$ and  
 $$h^u_{p,q}(z_0)= z_r$$ 
\end{itemize}

In this game, one sees that 
\begin{itemize}\item the points in $\tilde{\cX}$ induce a contraction of the horizontal coordinate of $z_s$, increasing the chances of the holomomy to be defined on $z_0$
\item \emph{a contrario} the points  in $\tilde{\cY}$ induce an expansion of the horizontal coordinate of $z_s$ making it in this way more likely to meet the positive stable manifold of new points in $\tilde \cX\cup\tilde \cY$. If the new points are in $\tilde{\cY}$ the expansion continues. This explains why, after surgeries, the quadrant $C_{+,+}(p)$ may be no more complete for $Y$.  This is what happens if $\tilde{\cX}$ is empty, which was already shown in \cite{Fe1}. 
\end{itemize}

When playing the previous game, an important tool emerges: if $2$ successive points $\gamma_i$ and $\gamma_{i+1}$ both belong to $\tilde \cX$ (resp. $\tilde \cY$), then there is a rectangle admitting $\gamma_i$ and $\gamma_{i+1}$ as corners, which is disjoint from $\tilde \cY$ (resp. from $\tilde \cX$).  This rectangle will be the main object of Section~\ref{s.rectangle}.  The existence or not of such rectangles is what determines the different cases that we consider in our study. 

\subsubsection{In the $C_{-,-}$ quadrants}

 The game in the $C_{-,-}$ quadrants is identical: crossing the negative stable manifold of a point in $\tilde{\cY}$ (resp. $\tilde{\cX}$)  induces an expansion (resp. contraction).

\subsubsection{In the $C_{-,+}$ and $C_{+,-}$ quadrants }

In the $C_{+,-}$ and $C_{-,+}$ quadrants, the descrition of the game is similar, but the roles of $\tilde{\cX}$ and $\tilde{\cY}$ are interchanged (the unique difference in the formulas is the sign before $\mu$ and $\nu$): 

\begin{itemize}
\item $t_i= u_i+\lambda^{+\mu(\gamma_i)}(u_i-t_{i-1})$ if $\gamma_i\in\tilde \cX$ 
\item $t_i= u_i+\lambda^{-\nu(\gamma_i)}(u_i-t_{i-1})$ if $\gamma_i\in\tilde \cY$  
\end{itemize}

Thus in these quadrants crossing the stable (positive or negative, according to the quadrant) separatrix of a point in $\tilde{\cY}$ induces a contraction and crossing the separatrix of a point in $\tilde{\cX}$ induces an expansion.

\section{Surgeries on the geodesic flow and $\RR$-covered Anosov flows}\label{s.geodesic}
The main goal of this section is to prove  Theorem~\ref{t.geodesic}: surgeries on a set of periodic orbits associated to disjoint simple closed geodesics do not change the bi-foliated plane.

We start by formulating a general criterion for preserving the $\RR$-covered character of Anosov flows after surgeries. 

\subsection{$\RR$-covered Anosov flows}
We remind that for any finite set of periodic orbits $\Gamma$, ${\cS}urg(X, \Gamma)$ denotes the set of Anosov flows obtained by $X$ by performing surgeries on $\Gamma$ up to orbital equivalence. In this section, $X$ is either a suspension or a positively twisted Anosov flow.  In other words, the bi-foliated plane $(\cP_X, F^s_X,F^u_X)$ is either trivial or conjugated to the restriction of the trivial (horizontal/vertical) foliations of $\RR^2$ to the strip $\{(x,y)\in \RR^2, |x-y|<1\}$. 

According to Corollary~\ref{c.caractR}, our hypothesis is equivalent to the following property:  for any $x\in\cP_X$ the quadrants $C_{+,+}(x)$ and $C_{-,-}(x)$ have the complete intersection  property (in other words they are complete). 

For every $x\in \cP_X$, let us denote 
$$\De_+(x)=\{y\in \cP_X, F^s_+(x)\cap F^u_-(y)\neq \emptyset \mbox{ and }F^u_+(x)\cap F^s_-(y)\neq \emptyset\}$$ 
$$\De_-(x)=\{y\in \cP_X, F^s_-(x)\cap F^u_+(y)\neq \emptyset \mbox{ and }F^u_-(x)\cap F^s_+(y)\neq \emptyset\}$$

These sets are included respectively in $C_{+,+}(x)$ and $C_{-,-}(x)$.  Our hypothesis is equivalent to the fact that for any $x\in \cP_X$, $\De_+(x)$ and $\De_-(x)$ are conjugated to the trivially bi-foliated plane.

By Propositions~\ref{p.caractR} and~\ref{p.caractR1} we obtain the desired criterion: 

\begin{coro}\label{c.caractR} Let $X$ be an Anosov flow which is $\RR$-covered positively twisted. Let $\Ga$ be a finite set of periodic orbits of $X$ 
 
Assume that for all  $x\in\tilde \Ga_X$ corresponding to $\gamma\in\Ga$ 
$$ \De_+(x)\cap \tilde\Ga_X=\emptyset=\De_-(x)\cap \tilde \Ga_X.$$

Then every $Y \in {\cS}urg(X,\Ga)$ is $\RR$-covered positively twisted. 
\end{coro}

\begin{prop}\label{p.caractR} Let $X$ be an Anosov flow which is either a suspension or $\RR$-covered positively twisted. Let $\Ga$ be a finite set of periodic orbits of $X$ and $Y \in {\cS}urg(X,\Ga)$.

Assume that for every $x\in \tilde \Ga_Y$ the quadrants $C_{+,+}(x)$ and $C_{-,-}(x)$ are complete. Then $Y$ is either a suspension or is $\RR$-covered positively twisted. 
\end{prop}
In other words, the completeness of the $C_{+,+}$ and $C_{-,-}$ quadrants at the points where we performed the surgery guarantees the completeness of every $C_{+,+}$ and $C_{-,-}$ quadrant.

\begin{prop}\label{p.caractR1} Let $X$ be an Anosov flow which is either a suspension or $\RR$-covered positively twisted. Let $\Ga$ be a finite set of periodic orbits of $X$ and $Y \in {\cS}urg(X,\Ga)$.
 
Assume that there is  $x\in\tilde \Ga_X$ corresponding to $\gamma\in\Ga$ such that 
$$ \De_+(x)\cap \tilde\Ga_X=\emptyset$$

Then for any $y\in\tilde\Ga_Y$ corresponding to $\gamma$ one has that 
$C_{+,+}(y)$ is complete. 
The same statement holds if we change $\De_+(x)$ to $\De_-(x)$ and $C_{+,+}(y)$ to $C_{-,-}(y)$. 
\end{prop}
\begin{proof} [Proof of Proposition~\ref{p.caractR1}]  This is a straightforward consequence of Theorem~\ref{t.newholonomies}, where the number $\ell$ in the statement is $1$: the positive unstable leaves $F^u_{X,+}(z)$ with $z\in F^s_{X,+}(x)$ do not meet any positive stable separatrix of an element in $\tilde \Gamma_X$. Therefore, the holonomies in $\De_+(x)$ are not affected at all by the surgeries on $\Gamma$. 
\end{proof}

Corollary~\ref{c.caractR} is a straightforward consequence of Propositions~\ref{p.caractR} and~\ref{p.caractR1}.  We are therefore only going to prove Proposition~\ref{p.caractR}.

\begin{proof}[Proof of Proposition~\ref{p.caractR}] Assume, by contradiction, that there is $x_Y\in \cP_Y$  such that the quadrant  $C_{+,+}(x_Y)$ (or $C_{-,-}(x_Y)$)  has not the complete intersection property.  In other words, there is $z_Y\in F^u_{Y,+}(x_Y)$ such that the holonomy $h^u_{Y,x_Y,z_Y}$ is not defined on the whole $F^s_{Y,+}(x_Y)$.  However, by transversality, the holonomy is defined on an open interval containing $x_Y$ so there exists a first point $y_Y$ on which the holonomy is not defined.

Consider the triple of corresponding points $x_X$, $z_X\in F^u_{X,+}(x)$, $y_X\in F^s_{X,+}(x_X)$. 

As the holonomy is not defined at $y_Y$, Theorem~\ref{t.newholonomies} implies that $F^u_{X,+}(y_X)$ intersects the positive stable separatrix of a point in $\tilde \Ga_X$ between  $F^u_{X,+}(x_X)$  and $F^u_{X,+}(y_X)$  (see figure \ref{f.proposition41}).  Let $\tilde \gamma_X$ be the first such point. There exists a rectangle $R$, whose interior is disjoint from $\tilde\Ga_X$ and whose boundary consists of $[x_X,y_X]^s$, a segment of $F^u_{X,+}(x_X)$,  a segment of $F^u_{X,+}(y_X)$ and a segment of $F^s_X(\gamma_X)$ containing $\gamma_X$.  Hence, there is a point $y_0$ in $F^u_X(\tilde \gamma_X)\cap F^s_{X,+}(x_X)$ belonging to $(x_X,y_X)^s$.   

This implies that the holonomy map $h^u_{Y,x_Y,z_Y}$  is defined on  $y_{0,Y}$, hence $F^u_{Y,+} (\tilde \gamma_Y)$ cuts  $F^s_{Y,+}(z_Y)$. 

In other words, both $F^u_{Y,+}(y_Y)$ and $F^u_{Y,+}(z_Y)$ eventually enter the quadrant $C_{+,+}(\tilde\gamma_Y)$.   By assumption, $C_{+,+}(\tilde \gamma_Y)$ is complete, so $F^u_{Y,+}(y_Y)\cap F^s_{Y,+}(z_Y)\neq \emptyset$ which contradicts our initial hypothesis.   
\begin{figure}[h!]
\includegraphics[scale=0.60]{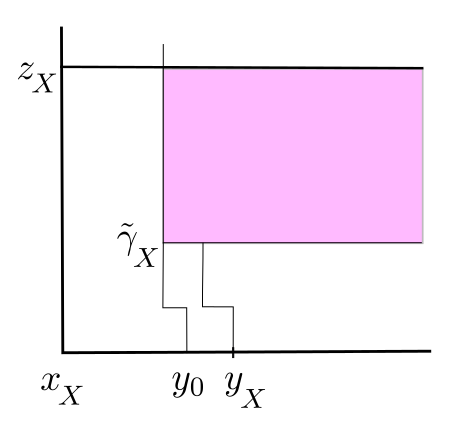}
\caption{}
\label{f.proposition41}
\end{figure}
\end{proof}

\subsection{The geodesic flow}

Theorem~\ref{t.geodesic} is now a straightforward consequence of the following proposition:

\begin{prop}\label{p.geodesic}Let $S$ be a hyperbolic surface. Let $\Ga=\{\gamma_1^+,\gamma_1,^-,\dots ,\gamma_k^+, \gamma_k^-\}$ be orbits of the geodesic flow $X$ corresponding to closed simple dijoints geodesics $(c_1,\dots, c_k)$.  Then for any $\tilde \gamma$ in $\tilde \Ga_X$ one has 

$$\De_+(\tilde \gamma)\cap \tilde\Ga_X=\emptyset=\De_-(\tilde \gamma)\cap \tilde \Ga_X.$$
\end{prop}

Propsition~\ref{p.geodesic} is itself a straighforward consequence of the next lemma: 
\begin{lemm}\label{l.geodesic} Consider $x\in\cP_X$ and $\gamma_x$ the corresponding geodesic in the Poincar\'e disk $\DD$.  Then $\De_{+}(x)$ is identified to the set of geodesics cutting transversely and positively $\gamma_x$. 
 
\end{lemm}
\begin{proof} The geodesics $\sigma\subset \DD$ crossing positively a given geodesic $\gamma\subset \DD$ are in one to one correspondance with the set of pairs of points $(\alpha(\sigma),\omega(\sigma))\in\partial\DD$, which belong in different connected components of $\DD\setminus\{ \alpha(\gamma),\omega(\gamma)\}$.  Now $\sigma$ is the unique intersection point between the stable manifold of the geodesic associated to $(\alpha(\gamma), \omega(\sigma))\in F^u(\gamma)$ and the unstable manifold of the geodesic associated to $(\alpha(\sigma), \omega(\gamma))\in F^s(\gamma)$. 
 
\end{proof}

\section{Surgeries preserving the branching structure of a non-$\RR$-covered Anosov flow}\label{s.nonR}

The aim of this section is to make an observation, which is almost clear after reading  \cite{Fe2}\footnote{ S. Fenley told us that, at the time he wrote \cite{Fe2} he was aware of this result, but did not publish it.  We thought that the present paper would may be a good place for publishing it.}. The results of this section generalise Theorems \ref{t.pivot0} and \ref{t.cS0}.

Given an Anosov vector field $X$, Fenley proved in \cite{Fe2} that:
\begin{enumerate}
 \item The leaves of $F^s_X$ which are not separated  correspond to finitely many periodic orbits of $X$.  Let us denote this set ${\cS}^s(X)= {\cS}^s_+(X)\cup{\cS}^s_-(X)$, where ${\cS}^s_+(X)$ and ${\cS}^s_-(X)$ correspond to the leaves, which are not separated from above or from below respectively.  The sets ${\cS}^s_+(X),{\cS}^s_-(X)$ are not necessarily disjoint. 
 \item Similarly, the leaves of $F^u_X$ which are not separated  correspond to a finite set of periodic orbits of $X$ denoted by
 $${\cS}^u(X)= {\cS}^u_+(X)\cup{\cS}^u_-(X).$$
 \item The  set of stable leaves in $\cP_X$, which are not separated from below from a given leaf $L^s_0$ are ordered as an interval of $\ZZ$, so let us denote them $\{L_i, i\in I\subset \ZZ\}$.  To each pair $L_i,L_{i+1}$ of successive non separated leaves, one can associate a point in $\cP_X$ called the pivot.
 Each pivot corresponds to a periodic orbit of $X$ (called a pivot periodic orbit) and the set of pivots is finite. Let us denote 
 $$Piv(X)=Piv^s_+(X)\cup Piv^s_-(X)\cup Piv^u_+(X)\cup Piv^u_-(X)$$ 
 the set of pivot periodic orbits of $X$. 

\end{enumerate}

The first observation is that performing surgeries on the pivots does not change at all the branching structure:

\begin{theo}\label{t.pivot} Let $X$ be a non-$\RR$-covered Anosov flow, $Piv(X)$ its set of periodic pivots and  $Y \in {\cS}urg(X,Piv(X))$. Then, under the natural identification of the orbits of $Y$ with the orbits of $X$, one has : 
$$Piv^{s/u}_\pm(Y)=Piv^{s/u}_\pm(X) \mbox{ and } \cS^{s/u}_{\pm}(Y)=\cS^{s/u}_{\pm}(X)  $$ 
 
\end{theo}

Furthermore, performing surgeries on the set of orbits corresponding to lower-non-separated stable leaves  may only change the upper-non-separated periodic stable leaves and their pivots, but not the other branching leaves or pivots: 

\begin{theo}\label{t.cS} Let $X$ be a non-$\RR$-covered Anosov flow and  $Y \in {\cS}urg(X,{\cS}^s_-(X))$.  Then, under the natural identification of the orbits of $Y$ with the orbits of $X$ one has : 
$$Piv(Y)^u_\pm=Piv^u_\pm(X) \mbox{ and } \cS^{u}_{\pm}(Y)=\cS^{u}_{\pm}(X)$$ 
and
$$Piv(Y)^s_-=Piv^s_-(X) \mbox{ and } \cS^{s}_{-}(Y)=\cS^{s}_{-}(X).$$ 
\end{theo}

The proof of both theorems is based on \cite{Fe2}. 
Fenley showed that, if $L_1$ and $L_2$ are consecutive stable leaves of $F^s_X$ not separated from below then:
\begin{itemize}\item there is $\gamma$ in $\pi_1(M)$ fixing both leaves $L_1$ and $L_2$. Each of those leaves contain a fixed point $x_i$ for the action of $\gamma$ on $\cP_X$. 
\item there is a proper embedding of $[-1,1]^2\setminus\{(-1,-1),(0,1),(1,-1)\}$ in $\cP_X$ uniquely associated to the pair $(L_1,L_2)$ and conjugating the trivial foliations with $F^s_X$ and $F^u_X$. 
 \item the image of this embedding is a \emph{pair of adjacent lozenges}, according to the terminology in \cite{Fe2} (see figure \ref{f.nonseparated}). The points $(-1,1),(1,1)$ correspond to $x_1,x_2$ and the point $(0,-1)$, whose positive unstable leaf ends at the missing point $(0,1)$ is called the \emph{pivot} associated to $L_1$, $L_2$ and is unique, hence also a fixed point of $\gamma$. 
\end{itemize}
\begin{figure}[h!]
\includegraphics[scale=0.60]{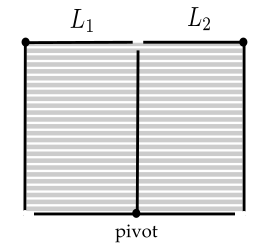}
\caption{}
\label{f.nonseparated}
\end{figure}

Now Theorem ~\ref{t.pivot} and Theorem~\ref{t.cS} result from the following lemma:

\begin{lemm}~\label{l.pivot}
 \begin{itemize}\item Every pivot is disjoint from the interior of any pair of lozenges associated to a pair of successive non-separated stable or unstable leaves. 
  \item  Let $x$ be a periodic point in a non-separated from below stable leaf.  If $x$ belongs to the interior of a pair of adjacent lozenges, then this pair is associated to a pair of stable leaves, which are not separated from above. 
 \end{itemize}

\end{lemm}

\begin{proof}
 Assuming by contradiction that a pivot of a pair of adjacent lozenges is contained in the interior of another pair, one checks that in every case one of the pairs needs to \emph{contain one of the missing points} of the other pair, which is impossible. 
 
 The same happens if one periodic point associated to a stable leaf non-separated from bellow belongs to the interior of the union of two adjacent lozenges associated to unstable non-separated leaves or to stable leaves non-separated from bellow. 
\end{proof}

\section{Domination of the contracting holonomies}\label{s.rectangle}

For the rest of this section, we fix a hyperbolic matrix $A\in SL(2,\ZZ)$ with positive trace and eigenvalues $\lambda, \lambda^{-1}$ satisfying $0<\lambda^{-1}<1<\lambda$ (once again similar arguments apply in the case where $A$ has negative eigenvalues). We denote by $X$ the Anosov flow which is the suspension of $f_A$ and its associated mapping torus $M=M_A$. We fix an orientation on the stable and unstable directions $E^s, E^u$ of $A$, which defines an orientation on the corresponding foliations on $\cP_X$.  

We begin by proving the Lemma \ref{l.orbitesfinies}, stating that for any finite $f_A$ invariant set $\cX\subset \TT^2$, there exist finitely many orbits of primitive $\cX$-rectangles, for the action of the group generated by $A$ and the integer translations.  
\begin{proof} [Proof of Lemma \ref{l.orbitesfinies}] We do the proof for positive $\cX$ rectangles. 

Using the $f_A$-invariance, one can choose an $\cX$-rectangle $R$ in each orbit so that the ratio between the lengths of the stable and unstable sides is contained in $[1,\lambda^2)$. 

Using the integer translations, one can also assume that the first point of the increasing diagonal of $R$ is in $[0,1)^2$. 

Consider the endpoint $e(R)$ in $X$ of the increasing diagonal of $R$. As $\tilde \cX$ is discrete, if the set of such rectangles $R$ is infinite one gets that $e(R)$ tends to infinity.  In this case, as the ratio of the lengths of the stable and unstable sides is bounded, $R$'s area also tends to infinity and as a consequence of this $R$ contains in its interior an arbitrary number of points in $\tilde \cX$, which contradicts the primitive assumption on $R$. 
\end{proof}

\subsection{If no $\cX$-rectangle is disjoint from $\tilde \cY$ the contracting holonomies dominate}
In Section~\ref{ss.game}, we presented the holonomies as a dynamical game, where \emph{crossing the positive stable separatrices of points in $\tilde \cX$ or in $\tilde \cY$ leads to either an expansion or a contraction}.  The holonomy will not always be defined when the expansion is strong. We also noticed that, in order to get two successive expansions, one needs to have a $\cX$-rectangle disjoint from $\tilde \cY$.  When no $\cX$-rectangle is disjoint from $\tilde \cY$, the expansion due to the points in $\tilde \cX$ can be neutralized by a sufficiently strong contraction associated to the points in $\tilde \cY$.  That is exactly what we prove in Theorem \ref{t.twisted}.

Theorem~\ref{t.twisted} consists in proving that the hypothesis \emph{no positive primitive  $\cX$-rectangle disjoint from $\tilde \cY$} implies that the contractions in the $C_{+,+}$ and $C_{-,-}$ quadrants due to (sufficiently strong) positive surgeries on $\cY$ dominate any surgery on $\cX$. 
The contractions in the $C_{+,-}$ and $C_{-,+}$ quadrants due to negative surgeries on $\cX$ cannot at the same time dominate the surgeries performed on $\cY$, which leads to a dynamical proof of the lemma \ref{l.asymetric} (which can also be proven geometrically).


If there are neither positive nor negative $\cX$-rectangles disjoint from $\tilde \cY$ then one gets: 

\begin{coro}Let $\cX,\cY$ be two finite $f_A$ invariant disjoint sets. Assume that there are no primitive $X$-rectangles disjoint from $\tilde \cY$. Then there is $N>0$ such that, if $Y \in {\cS}urg(X_A,\cX,\cY,\ast, (n_j)_{j\in J})$ where the $n_j$ are of the same sign and of absolute value greater than $N$, then $Y$ is $\RR$-covered twisted (positively or negatively, according to the sign of the $n_j$).
\end{coro}

This is particularly interesting in view of lemma \ref{l.epsilon}, which proves that the hypothesis \emph{no $\cX$-rectangle disjoint from $\tilde \cY$} holds frequently. More particularly, according to the lemma, for any $f_A$-invariant finite set $\cX$, there is $\varepsilon>0$ such that every $\varepsilon$-dense finite invariant set $\cY$ intersects every $\cX$-rectangle. 
\begin{proof}[Proof of Lemma \ref{l.epsilon}]
 This is a simple consequence of the fact that there are finitely many orbits of primitive $\cX$-rectangles. Fix a finite family of $\cX$-rectangles containing one rectangle of every orbit. 
 The lift on $\RR^2$ of an $\epsilon$-dense set $\cY$ will have a point in each of these finitely many rectangles, when $\epsilon$ is small enough.  The lift $\tilde\cY$ is invariant by integer translations. If furthermore $\cY$ is  $f_A$-invariant, one gets that the lift  $\tilde \cY$ contains a point in each primitive $\cX$-rectangle, hence in every $\cX$-rectangle.  
\end{proof}

\begin{rema}\label{r.epsilon} If in Lemma~\ref{l.epsilon} one chooses $\varepsilon>0$ very small, then $\tilde \cY$ will have an abundance of points in any $\cX$-rectangle and even in any $\frac 1K$-homothetic subrectangle for any arbitrary choice of $K>1$. 
\end{rema}

The frequency of crossing the separatrices of points in $\cY$ can counterbalance a possible lack of strength of the contractions associated to $\cY$, which is the basic idea behind the proof of Theorem \ref{t.Epsilon}.


Our aim from now on is to prove the Theorems~\ref{t.twisted} and~\ref{t.Epsilon}: assuming the lack of positive $\cX$-rectangles disjoint from $\tilde\cY$, strong positive surgeries along $\cY$ generate $\RR$-covered positively twisted Anosov flows. 

In order to do so, we will use the fact that a flow $Y$ is $\RR$-covered (trivially or positively twisted) if every $C_{+,+}$ and $C_{-,-}$ quadrant at every point of $\cP_Y$ is complete. Then we will discard the trivial case, getting the positive twist property. 

We start by proving the completeness at all points in $\tilde\cX$.

\subsection{Completeness of the quadrants $C_{+,+}(x)$ and $C_{-,-}(x)$ for  $x\in\tilde\cX$, for large surgeries on $\cY$.}

In this section we'll prove the following result:

\begin{prop}\label{p.complete} Assume that $\cX$ and $\cY$ are two finite invariant sets such that there are no positive $\cX$-rectangles disjoint from $\tilde \cY$. 
Then there is $N>0$ such that if $Y\in {\cS}urg(X_A,\cX,\cY,\ast, (n_j)_{j\in J})$ and $n_j> N$, then at every $x\in \tilde \cX$ the quadrant $C_{+,+}(x)$ is complete. The same holds for the quadrant $C_{-,-}(x)$. 
\end{prop}

Consider $x\in \tilde \cX$.  A \emph{primitive positive $(\cX,x)$-rectangle} is a primitive positive $\cX$-rectangle whose increasing diagonal has its origin on $x$. 

We identify $W^s_+(x)$ with $[0,+\infty)$ in an affine way. Any primitive positive $(\cX,x)$-rectangle $R$ is uniquely determined by its \emph{base segment} $[0,\mu_R]:=R\cap W^s_+(x)$. 

To every $t\in[0,+\infty)$ one associates the primitive  $(\cX,x)$-rectangle $R(t)$ with the largest base segment, which does not contain $t$ in its interior. The base segment of this rectangle is equal to $[0,\mu(t)]$, where 
$\mu(t)= \max\{\mu_R\leq t\}$. 

Similarly, one can define the primitive  $(\cX,x)$-rectangle with the smallest base segment containing $t$ in its interior and $\nu(t)=\min\{\mu(R)>t\}$. $\mu$ and $\nu$ are well defined thanks to the lemma \ref{l.orbitesfinies}.

One denotes $\rho(t)$ the positive number such that the point $(\mu(t),\rho(t))\in \tilde \cX\cap R(t)$ is the endpoint of the positive diagonal of the rectangle $R(t)$. 

By assumption on $\cX,\cY$, the rectangle $R(t)$ contains a point of $\tilde Y$ in its interior. We consider the smallest first coordinate of such a point; more precisely, 
we denote by $\delta(t)$ the smallest $r>0$ such that there is $(r,s)\in \tilde Y\cap R(t)$.  Clearly we have 
$$0<\delta(t)<\mu(t)$$

\begin{lemm}\label{l.complete}   Recall that  $\lambda>1$ is the expansive eigenvalue of $A$. 
 There is $N$ such that for every $t>0$ and every $n\geq N$ one has 
 
 $$\mu\left(\delta(t)+ \lambda^{-n}\left(t-\delta(t)\right)\right)=\mu(\delta(t)).$$
\end{lemm}

\begin{proof}
 Just notice that the functions $\mu,\delta$ are equivariant under the multiplication by $\lambda^{\pi}$, where $\pi$ is a common multiple of the periods of points in $\cX,\cY$. 
 
 Therefore we only need to prove the lemma for $t$ in the interval $[1,\lambda^\pi]$; therefore, for a finite number of intervals $[\mu(t),\nu(t))$. 
 
 For $n$ large enough $\delta(t)+ \lambda^{-n}\left(\nu(t)-\delta(t)\right)$ is very close to $\delta(t)$ (see figure \ref{f.prooflemma61}) and since the function $\mu$ is constant on an interval of the form $[\delta(t), \delta(t)+\epsilon]$, we get the desired result.  
\end{proof}

\begin{figure}[h!]
\includegraphics[scale=0.70]{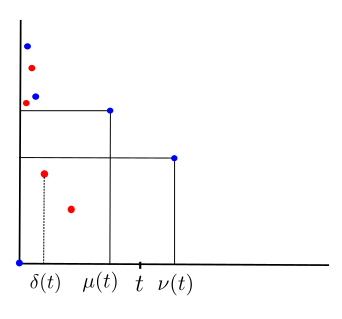}
\caption{In the above figure red points represent points in $\tilde \cY$ and blue ones points in $\tilde \cX$}
\label{f.prooflemma61}
\end{figure}
\begin{proof}[Proof of Proposition~\ref{p.complete}]  Fix $N$ given by lemma~\ref{l.complete}. Let $Y\in {\cS}urg(X_A,\cX,\cY,\ast, (n_j)_{j\in J})$ with $n_j> N$. Consider the positive unstable separatrix through the point $t$ (see figure \ref{f.proposition61}).  One follows it starting at $t$. 
The first point of $\tilde \cX\cap C_{+,+}(x)$ whose positive stable separatrix we meet is the point $x_0=(\mu(t),\rho(t))$.  However, before that, we meet the positive separatrices of all the points in $\tilde \cY\cap R(t)$ and maybe of some points of $\tilde Y$ outside $R(t)$.  The holonomy of the vector field $Y$ obtained by surgery consists in changing each time the intersection point by multiplying the distance to the point in $\tilde \cY$ by a factor $\lambda^{-n}$ (where $n$ is the product of the characteristic number by the period) which by hypothesis is smaller than $\lambda^{-N}$. Thus, according to Lemma~\ref{l.complete}, by playing the holonomy game we reach the stable manifold  of $x_0=(\mu(t),\rho(t))$ at a point $(t_1,\rho(t))$ with 
$$\mu(t_1)=\mu(\delta(t))<\mu(t).$$

In particular, $(t_1,\rho(t))$ is on the negative stable separatrix of $x_0$ and therefore is not affected by the surgery on that point. 

We proceed by following the positive unstable separatrix of the point $(t_1,\rho(t))$. Let $x_1=(\mu(t_1),\rho(t_1))$ be the first point of $\tilde X\cap C_{+,+}(x)$ whose positive stable separatrix meets the positive unstable separatrix of $(t_1,\rho(t))$.

Again, by assumption there are points of $\tilde Y$ in $R(t_1)$. In particular, since $\mu(t_1)=\mu(\delta(t))<\delta(t)$ the points in $\tilde Y\cap R(t_1)$ are not in $R(t)$. 

By playing the holonomy game, we will reach the stable manifold of $x_1=(\mu(t_1),\rho(t_1))$ at a point $(t_2,\rho(t_1))$ with $\mu(t_2)<\mu(\delta(t_1))<\mu(t_1)$.  Thus this point is on the negative sepatratrix of $x_1$  and is not affected by the surgery on $X$ and we proceed in the same way.

By finite induction, we obtain a primitive rectangle $R(t_i)$ in the orbit of $R(t)$ and after this the procedure will become periodic modulo iteration by a power of $A$. 
In particular, while we play the holonomy game, the positive unstable separatrix of $t$ will come closer and closer to $F^u_X(0,0)$. 

This shows that the positive unstable separtrix of $t$ intersects the positive stable separatrix of every point in the positive unstable separatrix of $(0,0)$ in $\cP_Y$. Therefore, the $(+,+)$-quadrant at the point $(0,0)$ is complete, which ends the proof. 

Notice that the notion of positive $X$-rectangle is the same for the quadrants $(+,+)$ and $(-,-)$.  Therefore the same argument proves the completeness of the quadrant $C_{-,-}(x)$ for  $x\in\tilde\cX$. 
\end{proof}
\begin{figure}[h!]
\includegraphics[scale=0.80]{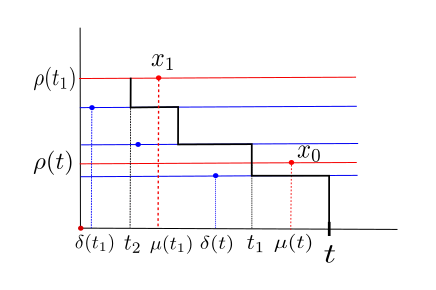}
\caption{In the above figure red points represent points in $\tilde \cX$ and blue ones points in $\tilde \cY$}
\label{f.proposition61}
\end{figure}
\subsection{Completeness of the quadrants $C_{+,+}(x)$ and $C_{-,-}(x)$ for  $x\in\tilde\cX$: replacing strong surgeries by the $\varepsilon$-density of $\cY$}

\begin{prop}\label{p.complete-epsilon} Assume that $\cX$ is a finite $f_A$-invariant set.  Then there is a $\varepsilon>0$ such that for any $\varepsilon$-dense finite $f_A$-invariant set $\cY$, one has the following property:  if $Y\in {\cS}urg(X_A,\cX,\cY,\ast, (n_j)_{j\in J})$ and $n_j>0$, then for every $x\in \tilde \cX$ the quadrants $C_{+,+}(x)$  and $C_{-,-}(x)$ are complete.
\end{prop}

Using the fact that the orbits of primitive  $\cX$-rectangles are finitely many, one gets $\varepsilon_0>0$ such that every $\varepsilon_0$-dense periodic orbit $\cY_0$ intersects every $\cX$-rectangle. 
We fix such a $\cY_0$. 

Fix $N>0$ given by Proposition~\ref{p.complete}. We denote by $\mu$ the product of $N$ by the period of $\cY_0$:
$$\mu=N\cdot \pi(\cY_0)$$

\begin{lemm}There is $\varepsilon>0$ such that, for every $\varepsilon$-dense $f_A$-invariant finite set $\cY$,  for any $x\in\tilde\cX$ and for any primitive positive $(\cX,x)$-rectangle $\cR$ there are at least $\pi(\cY_0)$  points $y\in \tilde\cY\cap \cR$ with the following properties: 
\begin{itemize}
 \item the period of $y$ is greater than $\mu$
 \item $y$ belongs to the connected component of $\cR\setminus \bigcup_{y_0\in\tilde \cY_0\cap R} F^u_X(y_0)$ which is bounded on one side by $F^u_X(x)$. 
 \item $y$ belongs to $\cR\setminus \bigcup \cR_i$, where the $\cR_i$ are the positive primitive $\cX,x$-rectangles having a strictly larger base. 
\end{itemize}
 
\end{lemm}
\begin{proof}We choose a rectangle $R$ in each orbit of primitive $\cX$-rectangles.  We just need to prove the statement for this finite collection of $(\cX,x)$-rectangles $\cR$. One still gets a finite collection of rectangles by considering  $\cR\setminus \bigcup \cR_i$ where $\cR_i$ are the positive primitive $(\cX,x)$-rectangles having a strictly larger base. 

The second announced property is granted by  the $\varepsilon$ density. Furthermore if $\varepsilon$ is sufficiently small, any periodic point $\varepsilon$-close to $F^u_X(x)\cap R$ has period greater than $\mu$.  
\end{proof}

\begin{proof}[Proof of Proposition~\ref{p.complete-epsilon}]. 
We'll compare the holonomy game for a vectorfield $Y$ obtained by positive surgeries on $\cY$ and the vector field $Y_0$ obtained by surgeries along $\cY_0$ of characteristic number $N$. More specifically, we will check that the holonomy of $Y$ is more contracting than the one of $Y_0$, which will finish the proof.

In the proof of Proposition~\ref{p.complete}, we defined by induction a sequence of points $x_i=(\mu(t_i),\rho(t_i))$, which are the successive points of $\tilde \cX\cap C_{+,+}(x)$ appearing in the game for $Y_0$.  

It's easy to check that under the hypotheses of the previous lemma, $$h^u_{Y,(0,0),x_0}(0,t)<h^u_{Y_0,(0,0),x_0}(0,t)$$In order to control the holonomy of $t$ it is therefore enough to control the holonomy of the point $(t_1,\rho(t))$.  
 
\end{proof}

\subsection{Completeness of the $(+,+)$-quadrants at every point}
The aim of this section is to deduce, from Proposition~\ref{p.complete} that every $(+,+)$ and $(-,-)$ quadrant  at any point of $P_Y$ is complete. 
\begin{prop}\label{p.complete-partout} Let $\cX,\cY$ de two disjoint finite $f_A$ invariant sets and any vector field $Y\in{\cS}urg(X_A,\cX,\cY,\ast, (n_j)_{j\in J})$ with $n_j\geq 0$. 

If for any $x\in \tilde \cX$ the quadrant $C_{+,+}(x)$ is complete, then for any $z\in\cP_Y$ the quadrant $C_{+,+}(z)$ is also complete.  
\end{prop}

As a straightforward consequence of Propositions~\ref{p.complete}, \ref{p.complete-epsilon} and \ref{p.complete-partout} one gets:

\begin{coro}\label{c.complete}
 With the hypothesis of Proposition~\ref{p.complete} or Proposition~\ref{p.complete-epsilon} the quadrants  $C_{+,+}(z)$ and  $C_{-,-}(z)$ are complete for any point $z\in\cP_Y$.
\end{coro}
\begin{proof}[Proof of Proposition~\ref{p.complete-partout}]
Assume by contradiction that there is $p\in \cP_Y$ for which the quadrant $C_{+,+}(p)$ is not complete. Therefore, there is a stable positive separatrix $W^s_+(r)$ with $r\in W^u_+(p)$ and $t_0\in W^s_+(p)$ such that the positive separatrix $W^u_+(t_0)$ does not intersect $W^s_+(r)$. 

Note that the set of unstable leaves cutting a stable leaf is open, as the intersections are transversal. Thus the leaves which don't cut $W^s_+(r)$ form a closed set, that doesn't contain $p$. Therefore, there is a smallest $t>0$ for which $W^u_+(t)$ does not cut $W^s_+(r)$.

If this unstable leaf does not cut the positive stable separatrix of a point in $\tilde X$ all the holonomies are defined and are smaller than $t$. Hence, by our initial hypothesis, there is $x\in\tilde X$ such that $W^u_+(t)$ cuts $W^s_+(x)$ and thus enters the quadrant $C_{+,+}(x)$.

If $W^s_+(r)$ crosses $W^u_+(x)$ (see figure \ref{f.corollary62}), then as the quadrant $C_{+,+}(x)$ is by assumption complete, it cuts $W^u_+(t)$ contradicting the hypothesis.  
\begin{figure}[h!]
\includegraphics[scale=0.55]{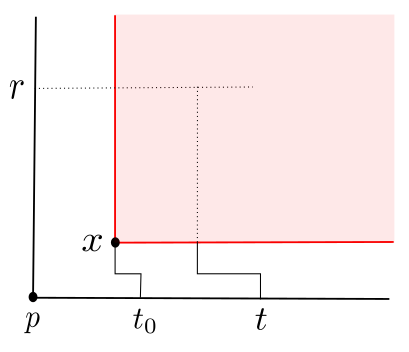}
\caption{}
\label{f.corollary62}
\end{figure}

Therefore, $W^s_+(r)$ cannot cross $W^u_+(x)$. Now the holonomy from $[0,t]$ to $W^s(x)$ is well defined and contains $x$ in its image: so $W^u_-(x)$ cuts $W^s_+(p)$ at some point $t_0<t$, which means that $W^s_+(t_0)$ does not cut $W^s_+(r)$. This contradicts the minimality of $t$, which ends the proof.
 
\end{proof}

\subsection{$\RR$-covered}
Using Corollary~\ref{c.complete} one gets

\begin{coro}\label{c.twisted}
 Under the hypotheses of Proposition~\ref{p.complete} or \ref{p.complete-epsilon} the flow $Y$ is $\RR$-covered and positively twisted. 
\end{coro}
\begin{proof} An Anosov flow for which every $(+,+)$ and $(-,-)$ quadrant is complete is $\RR$-covered and not negatively twisted. By our proofs of propositions \ref{p.complete} or \ref{p.complete-epsilon} the completeness of the $(+,+)$ and $(-,-)$ quadrants does not depend on the surgeries performed on $\cX$. However, if $Y$ were non-twisted, a negative surgery on $\cX$ would create a negatively twisted $\RR$-covered field. So $Y$ needs to be twisted. 
 
\end{proof}

This ends the proof of Theorems~\ref{t.twisted} and~\ref{t.Epsilon} and hence of Theorem~\ref{t.epsilon} too.

\subsection{On the existence of rectangles: proof of Lemma~\ref{l.asymetric}}
\begin{proof} Assume that every positive $\cX$-rectangle contains points in $\tilde\cY$ and every negative $\cY$-rectangle contains point in $\tilde \cX$. Then strong negative surgeries on $\cX$ and positive on $\cY$ would induce flows which are $\RR$-covered with all quadrants complete (according to Proposition~\ref{p.complete}), that is, with trivial bi-foliated planes. This persists under further negative surgeries on $\cX$, contradicting \cite{Fe1}. 
\end{proof}
\section{Domination of expanding holonomies: strings of $\cX$-rectangles disjoint from $\tilde\cY$}\label{s.string}
Our aim in this section on is to prove Theorems~\ref{t.nonRcovered}, \ref{t.notRcoveredatall} and \ref{t.string}. We start by proving Theorem \ref{t.string}.

Using the previous theorem, in order to prove Theorems \ref{t.nonRcovered} and \ref{t.notRcoveredatall}, we will need two things:
\begin{itemize}
\item The following theorem, which is going to be proved later in this section:  
\begin{theo}\label{t.separation} For any hyperbolic matrix $A\in SL(2,\ZZ)$ there are two periodic orbits $\cX$ and $\cY$ such that there exist positive $\cX$-rectangles disjoint from $\tilde \cY$ and negative $\cY$-rectangles disjoint from $\tilde\cX$. 
\end{theo}

\item To replace the hypothesis of large characteristic numbers in Theorem~\ref{t.string} by a large period as in the case of Theorems~\ref{t.epsilon} and \ref{t.Epsilon}.  
\end{itemize}





\subsection{Main step for Theorem~\ref{t.string}: undertwisted quadrants } 
Let $\cX$ be a finite $f_A$-invariant set. A \emph{string of positive $\cX$-rectangles} or a \emph{positive $\cX$-string} is a family of positive $\cX$-rectangles $R_i$ indexed by $\NN$ such that for any $k\in\NN$ the intersection $R_i\cap R_{i+1}$ is the endpoint of the increasing diagonal of $R_i$ and the initial point of the increasing diagonal of $R_{i+1}$.  The origin of the increasing diagonal of $R_0$ is called \emph{the origin of the string}. 

\begin{rema}
Let $\cX$ be a periodic orbit for $f_A$ and $\cY$ a finite $f_A$-invariant set disjoint from $\cX$. The existence of a positive (resp. negative) $\cX$-string with origin $x$ disjoint from $\tilde \cY$ is equivalent to the existence of a positive (resp. negative) $\cX$-rectangle disjoint from $\tilde \cY$.
\end{rema}

Theorem~\ref{t.string} is a consequence of the following technical result:
\begin{theo}\label{t.incomplete}Let $A\in SL(2,\ZZ)$ be a hyperbolic matrix, $\cX$ and $\cY$ finite $f_A$-invariant sets such that there exists a positive $\cX$-string disjoint from $\tilde \cY$. 

Then there is $n>0$ such that for any $Y \in {\cS}urg(X_A,\cX,\cY,(m_i)_{i\in I}, \ast)$ with $m_i<-n$, there is $x\in\tilde \cX$ such that the quadrants $C_{+,+}(x)$ and $C_{-,-}(x)$ are incomplete (i.e. undertwisted). 
\end{theo}
The important point in the statement of Theorem~\ref{t.incomplete} is that the conclusion does not depend on the surgeries performed (or not) on the orbits of the points in $\cY$.  

\begin{proof}[Proof of Theorem~\ref{t.string} assuming Theorem~\ref{t.incomplete}]
The hypotheses of Theorem~\ref{t.string} allow us to apply Theorem~\ref{t.incomplete} for the quadrants $C_{+,+}$ and $C_{-,-}$, but also the quadrants $C_{+,-}$ and $C_{-,+}$ by exchanging the roles of $\cX$ and $\cY$.   Therefore there is $n>0$ such that for all surgeries on $\cX$ and $\cY$ with negative characteristic numbers on $\cX$ and positive on $\cY$, all of absolute value greater than $n$, there exist a point $x\in\tilde \cX$ and a point $y\in\tilde\cY$  such that the quadrants 
$C_{+,+}(x)$ and $C_{+,-}(y)$ are both undertwisted.  This implies that $Y$ is not $\RR$-covered, which finishes the proof. 
\end{proof}

\subsection{Proof of Theorem~\ref{t.incomplete}: $\cX$-staircase disjoint from $\tilde\cY$}

\par{We consider $\RR^2$ endowed with a (linear) base in which $\cF^s$ is the horizontal foliation, $\cF^u$ is the vertical foliation and the origin $(0,0)$ belongs to $\tilde X$. We denote by $\lambda>1$ and $\lambda^{-1}<1$ the two eigenvalues of $A$. Again, if $A$ has negative eigenvalues our following arguments still apply. Fix $m$ to be the twist function for $X$ and $n$ the twist function for $Y$, that is the product of the characteristic number by the period for $f_A$. We have that $m(f_A(x))= m(x)$ and $n(f_A(x))= n(x)$.}

Recall that the stable and unstable foliations are oriented.  Hence, every rectangle $R$ has a top side $\partial^{s,up}R$, a bottom side $\partial^{s,low}R$, a right side $\partial^{u,right}R$ and a left side $\partial^{u,left} R$

A \emph{vertical }(resp. \emph{horizontal})  \emph{subrectangle} of $R$ is a rectangle $R_0\subset R$ such that $\partial^u(R_0)\subset \partial^u(R)$ (resp. $\partial^s(R_0)\subset\partial^s(R)$). 

Furthermore, we'll say that a horizontal subrectangle is a \emph{right horizontal subrectangle} if 
$$\partial^{u,right}(R_0)=\partial^{u,right}(R).$$

Given a rectangle $R$ we denote by $\ell^s(R)$ the length of its stable (top or bottom) sides and $\ell^u(R)$ the length of the 
unstable (right or left) sides.

\begin{defi}\label{d.staircase}
Fix $x\in\tilde X$. We say that an infinite sequence of rectangles $\cR_0,\cR_1,....$ is \emph{a positive $\cX-$staircase with origin at $x\in\tilde X$  in $C_{+,+}(x)$} (or a \emph{$(\cX,x,+,+)$-staircase}) if the following conditions are satisfied (see figure \ref{f.staircase}):
\begin{enumerate}
\item All rectangles $\cR_n$ are contained in $C_{+,+}(x)$.
\item There is a positive  $\cX$-string $\{\De_i\}_{i\in\NN}$ with origin at $x$ such that $\De_i$ is a right horizontal subrectangle of $R_i$ for every $i>0$ and $R_0=\De_0$.  
\item $ \partial^{s,up}\cR_m \subset \partial^{s,low} \cR_{m+1}$
\item $$\frac{\ell^s(\De_{i+1}) }{\ell^{s}(\De_{i})} \mbox{ is bounded for }i\in\NN.$$
\end{enumerate}

In a similar way, one can define a positive staircase inside $C_{-,-}(x)$ and negative staircases in  $C_{+,-}(x)$ and $C_{-,+}(x)$. 
\end{defi}

\begin{rema} By definition, the left unstable sides $\partial^{u,left}(R_i)$ of all the rectangles $R_i$ in a positive $(\cX,x,+,+)$-staircase $\cR=\{R_i\}_{i\in\NN}$ are segments on $F^u_+(x)$ which are adjacent (disjoint interior but sharing an endpoint with the next one), whose the union is an interval $\cI^u(\cR)$ on $F^u_+(x)$ starting at $x$:
$$\cI^u(\cR)=\bigcup_{i=0}^\infty \partial^{u,left}(R_i).$$

We say that $\cI^u(\cR)$ is \emph{the axis} of the staircase $\cR$
\end{rema}

\begin{rema} One also has that
$$\ell^s(R_i)=\sum_{j=0}^i\De_j.$$
\end{rema}

\begin{figure}[h!]
\includegraphics[scale=0.60]{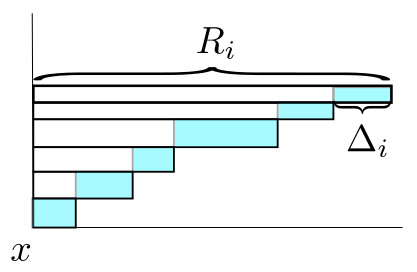}
\caption{}
\label{f.staircase}
\end{figure}

Theorem~\ref{t.incomplete} is a simple consequence of the next two lemmas:

\begin{lemm}\label{l.exists staircase}
Using the previous notations, if there is a positive $\cX$-string disjoint from $\tilde \cY$ with origin at $x\in\tilde \cX$, then there exists a positive $\cX$-staircase disjoint from $\tilde\cY$ with origin at $x$ inside $C_{+,+}(x)$.
\end{lemm}
\begin{lemm}\label{l.stairacase implies incomplete}
If for some point $x \in \tilde{\cX}$ there exists a positive $\cX$-staircase $(R_i)_{i \in \mathbb{N}}$ disjoint from $\tilde \cY$ with origin at $x$ in $C_{+,+}(x)$, then there exists $N'>0$ such that for any $Y \in {\cS}urg(X_A,\cX,\cY,(m_i)_{i\in I},\ast)$ with $m_i<-N'$ the $C_{+,+}(x)$ quadrant for $Y$ is undertwisted (i.e. incomplete).
\end{lemm}

\subsection{$\cX$-strings and $\cX$-staircase disjoint for $\tilde{\cY}$}

\begin{proof}[Proof of Lemma~\ref{l.exists staircase}]
Let $\tilde \cX$ and $\tilde \cY$ be two finite invariant sets and assume that $\{\De_i\}_{i\in\NN}$ is a positive $\cX$-string  disjoint from $\tilde{\cY}$ with origin at a point $x_0\in\tilde\cX$.  We can assume without loss of generality that the $\De_i$ are primitive. 

For any $i$ there is a unique rectangle, denoted by $D_i$, such that 
\begin{itemize}
 \item $\De_i$ is a right horizontal subrectangle of $D_i$
 \item the interior of $D_i$ is disjoint from $\tilde\cY$
 \item $D_i$ contains a point of $\tilde\cY$ on its left unstable side $\partial^{u,left}D_i$.
\end{itemize}
In other words, starting with $\De_i$ we push to the left its left unstable side until we touch a point in $\tilde \cY$ for the first time.

\begin{clai} There are $1<c<C$ such that for every $i$ one has 
$$c<\frac{\ell^s(D_i)}{\ell^s(\De_i)} <C$$
\end{clai}
\begin{proof} There are finitely many orbits of primitive $\cX$-rectangles and therefore there are also finitely many orbits of associated rectangles $D_i$.  The ratio in the statement is invariant under the action of $A$ and of integer translations, which gives the desired result. 
\end{proof}

We denote by $x_i^0$ the origin of $\De_i$ for every $i\in \NN$ ($x_0=x_0^0$).  By assumption all the $x_i^0$  belong to $\tilde \cX$.

As $\cX$ is finite and $f_A$-invariant, every point is periodic. 
Let $n>0$ denote a common period of all points in $\cX$ and $T_z$ be the translation by $z\in \tilde \cX$ in $\cP_{X_A}$. By definition of $n$, for every $z\in \tilde \cX$, $\phi_z:=T_z\circ A^n$ is the (unique) lift of $f_A^n$  having $z$ as a fixed point. $\phi_z$ is affine and its derivative on $\RR^2$ is $A^n$ in the canonical coordinates. In the $F^s_X,F^u_X$ coordinates, the derivative of 
$\phi_z$ is 
$$
\left(
\begin{array}{cc}
 \lambda^{-n}&0\\
 0&\lambda^n
\end{array}
\right)
$$

We are ready to define by induction the rectangles $R_i$ of the staircase and we start by fixing $R_0=\De_0$. 

Assume that $R_i$ has been defined and denote by $x_{i+1}$ the endpoint of its increasing diagonal.  Consider $g_i$ in the group generated by $A$ and the integer translations such that $x_{i+1}= g_i(x_{i+1}^0)$. 

Consider $g_i(D_{i+1})$. It is a rectangle whose bottom stable side contains $x_{i+1}$. We consider the orbit of $g_i(D_{i+1})$ by $\phi_{i+1}:=\phi_{x_{i+1}}$, which consists of rectangles containing $x_{i+1}$ in their bottom stable side. 

\begin{clai} For large $k>0$,  $\phi_{i+1}^k(g_i(D_{i+1}))$ is disjoint from 
$F^u_X(x_0)$ and $\phi_{i+1}^{-k}(g_i(D_{i+1}))$ is not disjoint from $F^u_X(x_0)$.
\end{clai}

Let $k_{i+1}=\min\{k| \phi_{i+1}^{-k}(g_i(D_{i+1}))\cap  F^u_X(x_0)\neq \emptyset\}$ and
 $$h_{i+1}=\phi_{i+1}^{-k_{i+1}}\circ g_i$$

By construction (see figure \ref{f.lemma71}) $F^u_X(x_0)$ cuts  $h_{i+1}(D_{i+1})$ in two horizontal sub-rectangles: $R_{i+1}$ is the right subrectangle. Notice that $R_{i+1}$ is disjoint from $\tilde \cY$ as it's included in the image of $D_i\setminus\partial^{u,left}(D_i)$ by $h_{i+1}$.

This defines by induction a family of rectangles $\{R_i\}$ satisfying all the conditions of  Definition~\ref{d.staircase} except possibly (4). 
\begin{figure}[h!]
\includegraphics[scale=0.60]{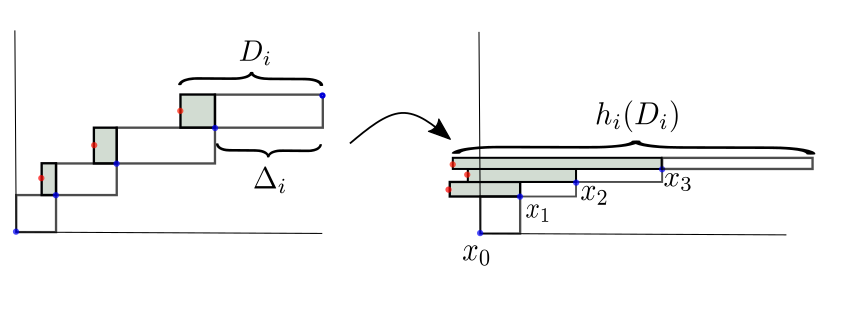}
\caption{}
\label{f.lemma71}
\end{figure}
It remains to check that $\frac{\ell^s(R_{i+1}) }{\ell^{s}(R_{i})}$ is bounded. Recall that $R_{i+1}$ is a right horizontal subrectangle of $h_{i+1}(D_{i+1})$ and that $D_{i+1}$ admits $\De_{i+1}$ as a right horizontal subrectangle.  Let us denote by $a_{i+1}=\ell^s(\De_{i+1})$ and $b_{i+1}=\ell^s(\tilde\De_{i+1})$ where   $\tilde \De_{i+1}=\overline{D_{i+1}\setminus \De_{i+1}}$. 

Because of the invariance of the ratio $\frac{a_i}{b_i}$ under the action of integer translations or $A$ and thanks to the finiteness of orbits of primitive $\cX$-rectangles, one gets that $\frac{a_i}{b_i}$ and $\frac{b_i}{a_i}$ are bounded.  

Now recall that $h_{i+1}$ is obtained by composing $g_i$ with $\phi^{-k_{i+1}}_{i+1}$, where $k_{i+1}$ is the minimal integer $k$ for which $\phi_{i+1}^{-k}\circ g_i(D_{i+1})$ meets $F^u_+(x_0)$. This implies (see figure) \ref{f.lemma71} that 

$$\sum_{j=0}^{i} \ell^s(h_{j}(\De_j))<\ell^s(h_{i+1}(\tilde \De_{i+1}))<\lambda^n. \sum_{j=0}^{i} \ell^s(h_{j}(\De_j))$$.

Let $\ell_i=\ell^s(h_i(\De_i))$ and $\tilde \ell_i=\ell^s(h_i(\tilde\De_i))$.  Then $\frac{\ell_i}{\tilde\ell_i}=\frac{a_i}{b_i}$ (because $h_i$ is affine), so this ratio and its inverse are bounded. 

By the previous inequality we have that $\frac{\tilde \ell_{i+1}}{\sum_{j=0}^i \ell_j}\in[1,\lambda^n]$. Hence, there is $C>0$ so that $\frac{\ell_{i+1}}{\sum_{j=0}^i \ell_j}\in[\frac1C,C]$. Finally, using the fact that $\ell^s(R_i)=\sum_{j=0}^i \ell_j$ we have that $\frac{\ell^s(R_{i+1})}{\ell^s(R_i)}$ is bounded from above. Therefore, $(R_i)_{i\in \mathbb{N}}$ verifies property (4) of the definition of a staircase and is indeed a positive $\cX$-staircase disjoint from $\tilde \cY$.
\end{proof}

\begin{rema} If $\cY$ is a non empty finite invariant set then for any 
finite invariant $\cX$ any $\cX$-staircase $\cR=\{R_i\}$ disjoint from $\tilde \cY$ has an axis of bounded length:
$$\ell(\cI^u(\cR))<\infty.$$
 
\end{rema}

\subsection{Staircase and the holonomy game: proof of lemma \ref{l.stairacase implies incomplete}}

 In this small section, we give the proof of lemma \ref{l.stairacase implies incomplete}, thus completing the proof of theorems \ref{t.incomplete} and \ref{t.string}. The proof is going to be the result of three fairly simple observations given in the form of lemmas.
 
 Let $\cR=\{R_i\}$ be a $(\cX,x,+,+)$-staircase disjoint from $\tilde \cY$ for some $x\in\tilde \cX$.
 For any $i$ we denote by $R_{i,\cY}$ the unique rectangle with the following properties:
 \begin{itemize}
  \item $\partial^{u,left} R_{i,\cY}=\partial^{u,left}R_i\subset F^u_+(x)$
  \item $R_{i,\cY}\cap\tilde{\cY}=\partial^{u,right}R_{i,\cY}\cap\tilde{\cY}\neq \emptyset$
 \end{itemize}
 In other words, one pushes the right side of $R_i$ to the right until it intersects $\tilde{\cY}$ for the first time. 
 
 We denote $S_{i,\cY}=\overline{R_{i,\cY}\setminus R_i}$. It is a right horizontal subrectangle of $R_{i,\cY}$ called the \emph{$\cY$-safe zone of $R_i$}. 
 
Once again, because of the finiteness of the number of orbits of $\cX$-rectangles, one gets that the ratio $\frac{\ell^s(S_{i,\cY})}{\ell^s(\De_i)}$ takes a finite number of values (in particular this ratio and its inverse are bounded).

For any $i$ let 
$$q_i=\partial^{s,low}R_i\cap F^u_+(x) \mbox{ and } q=\lim q_i\in F^u_+(x)$$

For any $Y\in \cS urg(A,\cX,\cY)$ we denote by 
$h_{i,Y}\colon F^s_+(q_i)\to F^s_+(q_{i+1})$ the holonomy of $F^u_Y$.

\begin{lemm} Using the above notations one has
\begin{itemize}
 \item if $(t,q_i)\in\partial^{s,low}(R_i)$ then 
 $$h_{i,Y}(t,q_i)=(t,q_{i+1})$$
 \item if $(t,q_i)\in\partial^{s,low}(S_{i,\cY})$ then
 $$h_{i,Y}(t,q_i)=(t_{i+1}+ \lambda^{-\tau(x_{i+1})}(t-t_{i+1}),q_{i+1})$$
 where $x_{i+1}=(t_{i+1},q_{i+1})=\partial^{s,low}(R_{i+1})\cap \tilde \cX$ and $\tau(x_{i+1})=m(x_{i+1})\cdot \pi(x_{i+1})$ is the twist number associated to $x_{i+1}$ , where $m(x_{i+1})$ is the characteristic number of the surgery at the orbit corresponding at $x_{i+1}$ and $\pi(x_{i+1})$ is its period. 
\end{itemize} 
\end{lemm}
\begin{proof}Just notice that for point in $\partial^{s,low}(R_i)$ their positive unstable leaf reaches $F^s_+(q_{i+1})$ without crossing any positive stable leaf of a point in $\tilde \cX\cup \tilde \cY$ at the right of $F^u(x)$, so the holonomy is not affected by the surgeries. 
For the points in $\partial^{s,low}(S_i)$, the unique moment when they cross a positive stable leaf of a point in $\tilde \cX\cup \tilde \cY$ at the right of $F^u(x)$ is precisely when they reach $F^s_+(q_{i+1})$: they cross the positive stable leaf of $x_{i+1}$ leading to the announced formula. 
\end{proof}

\begin{lemm}If for every $i$ we  have 
$$\partial^{s,low}(R_{i+1,\cY})\subset h_{i,Y}(\partial^{s,low}(R_{i,\cY}))$$
then $C_{+,+}(x)$ is undertwisted. 
\end{lemm}
\begin{proof}In this case the image by the holonomy of $F^u_Y$ of 
$\partial^{s,low}(R_{0,\cY})$ on $F^s_+(q_i)$ contains the segment $\partial^{s,low}(R_{i,\cY})$ whose length tends to infinity. 
Thus the holonomy from $F^s_+(x)$ to $F^s_+(q)$ takes $\partial^{s,low}(R_{0,\cY})$ to the whole $F^s_+(q)$, so the domain of that holonomy is contained in $\partial^{s,low}(R_{0,\cY})$, which finishes the proof. 
\end{proof}

Recall that the ratios $\frac{\ell^s(\De_{i+1})}{\ell^s(\De_i)}$, $\frac{\ell^s(S_{i+1})}{\ell^s(\De_{i+1})}$, 
$\frac{\ell^s(\De_i)}{\ell^s(S_i)}$ are bounded, therefore there is $C>0$ such that for every $i$ one has 
$$\frac{\ell^s(\De_{i+1})+\ell^s(S_{i+1})}{\ell^s(S_i)}<C$$

\begin{lemm}Let for all $x\in \cX$, $m(x)$ be the characteristic number of the surgery associated to $x$. Assume that all the $m(x)$ are negative and of large absolute value so that the product $\tau(x)=m(x)\pi(x)$ (where $\pi$ is the period function) satisfies for every $x\in\cX$ 
$$\lambda^{|\tau(x)|}>C$$

Then for every $i$ one gets $\partial^{s,low}(R_{i+1,\cY})\subset h_{i,Y}(\partial^{s,low}(R_{i,\cY}))$ and by the previous lemma $C_{+,+}(x)$ is incomplete.  
\end{lemm}

An example of a holonomy game that satisfies the hypotheses of the previous lemma is given in figure \ref{f.lemma75}.
\begin{figure}[h!]
\includegraphics[scale=0.60]{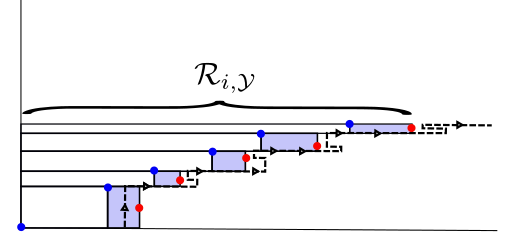}
\caption{In the above picture red (resp. blue) points represent points in $\tilde \cY$ (resp. $\tilde \cX$), white rectangles represent the staircase $\cR_i$ and the blue one the safety zones $S_i$. }
\label{f.lemma75}
\end{figure}
By combining the three previous lemmas, we obtain the proof of Lemma~\ref{l.stairacase implies incomplete}.

\subsection{Abundance of pairs $(X,Y)$ with strings of rectangles} 

As the proof of theorem \ref{t.string} has reached its end, we continue by proving theorem \ref{t.separation}, as it was announced in the beginning of section 7. In fact, in this small section we prove a much stronger result, resembling a lot Theorem~\ref{t.notRcoveredatall}. 

Consider two distinct periodic orbits $\cX$ and $\cY$ and the points $x\in\tilde \cX$ and $y\in\tilde \cY$. We remind that
\begin{rema}\label{r.existencestrings}
 If there is a positive $\cX$-rectangle $R$ disjoint from $\tilde \cY$, then there is a a positive $\cX$-string  disjoint from $\tilde \cY$ based at $x$.
\end{rema}

In order to prove the next corollary, we will use the following classical fact from ergodic theory. 

\begin{lemm}\label{l.basic} Let $f$ be a diffeomorphism of a compact surface, $p$ a hyperbolic periodic saddle point and $q_1,\dots,q_k$ transverse homoclinic intersections between a stable separatrix in $W^s(Orb(p))$ and an unstable separatrix in $W^u(Orb(p))$. Denote by $K$ the union of the orbit of $p$ and of the orbits of the $q_i$, which is an invariant compact set. 

Then for any neighbourhood $U$ of $K$, there is a hyperbolic basic set $\La_U$
(i.e. transitive and with local product structure) satisfying $K \subsetneq \La_U \subsetneq U$
\end{lemm}

We're ready to prove the main result of this section.

\begin{coro}\label{c.notRcoveredatall} Let $\cE\subset \TT^2$ be a finite $f_A$-invariant set.  Then there are periodic orbits $\gamma_+$, $\gamma_-$ such that there exist positive and negative $\gamma_+,\gamma_-$-rectangles disjoint respectfully from $\gamma_-\cup \cE$ and $\gamma_+\cup \cE$. 

Hence, as a result of theorem \ref{t.incomplete} there exists $n>0$ such that every flow $Y$ obtained from $X_A$ by surgeries on $\cE$ and by two surgeries of distinct signs along $\gamma_+$ and $\gamma_-$ with characteristic numbers of absolute value greater than $n$, is not $\RR$-covered. 
\end{coro}
\begin{proof} Choose $\sigma_+ \notin \cE$ and $\sigma_- \notin \cE$ two distinct periodic points and $q_{i,\pm}$, $i=1\dots 4$, four homoclinic intersections between the two stable and the two unstable separatrices of $\sigma_\pm$. We denote by $K_\pm$ the compact obtained by the union of $\sigma_\pm$ and the orbits of the $q_{i,\pm}$. 

We choose neighbourhoods $U_\pm$ of $K_\pm$ such that $U_+\cap U_{-}=U_+\cap \cE= U_{-}\cap \cE= \emptyset$. We denote by $\La_\pm$ the hyperbolic basic sets contained in $U_\pm$ and containing $K_\pm$ given by Lemma~\ref{l.basic}. 

There is $\varepsilon>0$ such that every periodic orbit $\gamma_+$ in $\La_+$, which is $\varepsilon$-dense in $\La_+$ admits positive and negative  $\gamma_+$-rectangles disjoint from $\cE$ and from $\La_-$.  By remark \ref{r.existencestrings}, there exist positive and negative $\gamma_+$-strings disjoint from $\cE$ and $\La_-$. 

In the same way, for $\varepsilon>0$ small enough, any periodic orbit $\gamma_-\subset \La_-$, which is $\varepsilon$-dense admits positive and negative strings disjoint from $\cE$ and from $\La_+$. 

Then Theorem~\ref{t.incomplete} implies that, large negative surgeries along $\gamma_+$ induce incomplete $C_{+,+}$ quadrants (at any point of the bi-foliated plane associated to $\gamma_+$) independently of the surgeries we perform on $\gamma_-\cup \cE$.  In the same way large positive surgeries along $\gamma_-$ induce incomplete $C_{+,-}$ quadrants at any point associated to $\gamma_-$.

Therefore, by performing large positive surgeries along $\gamma_-$ and large negative surgeries along $\gamma_+$ one obtains a non $\RR$-covered flow $Y$, independently of the surgeries performed along $\cE$. 
\end{proof}

In order to prove Theorems~\ref{t.nonRcovered} and \ref{t.notRcoveredatall}, we would like to remove the ``large enough'' hypothesis in Corollary~\ref{c.notRcoveredatall}.

\subsection{Replacing large characteristic numbers by large periods}
The aim of this section is to go from the proof of Theorems \ref{t.incomplete} and \ref{t.string} to the proof of Theorems \ref{t.nonRcovered} and \ref{t.notRcoveredatall}.  As Theorem~\ref{t.notRcoveredatall} clearly implies Theorem~\ref{t.nonRcovered}, we will only prove Theorem~\ref{t.notRcoveredatall}, that is, for any  finite set $\cE$ of periodic orbits there exist two periodic orbits $\gamma_+$ and $\gamma_-$ such that any $Y\in\cS urg(X_A,\cE\cup \gamma_+\cup\gamma_-)$ for which the surgeries along $\gamma_+$ and $\gamma_-$ are of different signs, is not $\RR$-covered. 

\subsubsection{Choosing a staircase and a safety zone}

As in the proof of Corollary~\ref{c.notRcoveredatall} we first 
build two disjoint hyperbolic basic sets $\La_+$ and $\La_-$ in $\TT^{2}$ that don't intersect $\cE$.  Then, for $\varepsilon_0>0$ sufficiently small, any $\varepsilon_0$-dense (in $\La_\pm$) periodic orbit  $\sigma_\pm\subset \La_\pm$ admits a positive and a negative $\sigma_\pm$-string
disjoint from $\tilde \cE \cup \tilde \La_\mp$. 

\begin{rema}\label{r.boundary}A classical fact in hyperbolic dynamical systems on surfaces is that hyperbolic  basic sets $\La$ admit at most finitely many \emph{periodic boundary points} (see for instance \cite{BL}), that is, points which are not accumulated by points in $\La$ in each of their $4$ quadrants.  

Therefore, if $\varepsilon$ is taken very small, the orbits $\sigma_+$ and $\sigma_-$ are not boundary periodic points of $\La_+$ and $\La_-$, respectively.  
\end{rema}

In this case, one could build in each quadrant $C_{\pm\pm}(\sigma_+)$ positive or negative (according to the quadrant) staircases for $\sigma_+$ disjoint from $\tilde \cE \cup \tilde\La_\mp$. In order to avoid heavy notations, we will just consider the $\sigma_+$-staircase $\{R_i\}$ in the $C_{+,+}$ quadrant and its associated positive $\sigma_+$-string $\{\De_i\}$.  Finally, we extend the staircase by its safety zone $S_i$ associated to $\cE\cup \La_-$. 

In the same way, one may construct in each quadrant of $\sigma_-$ a staircase disjoint from $\cE\cup\La_+$ and its safety zone. 

So, until now in each quadrant $C_{+,+}(x)$ with $x\in \tilde{\sigma_+}$ (where $\tilde{\sigma_+}$ is the lift of $\sigma_+$ in the bi-foliated plane) we've constructed the following:
\begin{itemize}\item $\cR=\{R_i\}$ the rectangles of a staircase disjoint from $\cE\cup \La_-$, whose left unstable sides are adjacent segments on $F^u_+(x)$, whose union is a bounded interval $\cI^u(\cR)$.  The ratio $\frac{\ell^s(R_{i+1})}{\ell^s(R_i)}$ is bounded and bounded away from $1$. 
 \item $(\De_i)$ a positive $\sigma_+$-string with origin at $x$. The $\De_i$ are right horizontal subrectangles of the $R_i$ and the ratios $\frac{\ell^s(R_{i})}{\ell^s(\De_i)}$, $\frac{\ell^s(\De_{i+1})}{\ell^s(\De_i)}$ and their inverses are bounded. Once again we may assume that the $\De_i$ are primitive $\sigma_+$-rectangles.
 \item the rectangles $S_i$, whose left unstable side is the right stable side of the $R_i$ and $\De_i$. Their intersection with $\cE\cup\La_+$ is contained in their right side and finally the ratio $\frac{\ell^s(R_i)}{\ell^s(S_i)}$ is bounded with bounded inverse.  
\end{itemize}

\subsubsection{Choosing the periodic orbits $\gamma_+$ and $\gamma_-$}

Note that $\partial^{s,low}(\De_{i+1})$ and $\partial^{s,up}(S_i)$ are two segments in the same stable leaf $F^s_+(q_{i+1})$ that are adjacent to the same segment  $\partial^{s,up}(\De_i)$. 
  
Furthermore, their lengths have a bounded ratio $\frac{\ell^s(\De_{i+1})}{\ell^s(S_i)}$. 
  
For every $i$  and every $\rho\in(0,1)$ we define by  
 $J_{i+1,\rho}\subset \partial^{s,low}(\De_{i+1})$, the segement adjacent to
$\partial^{s,up}(\De_i)$ of length
$$\ell(J_{i+1})=\rho\cdot\ell^s(\De_{i+1}).$$

Then there is $0<\rho<1$ small enough with the following property: 
$$J_{i+1,2\rho}\subset \partial^{s,up}(S_i) \text{  for all }i$$
 
Let $\De_{i,\rho}$ denote the left horizontal subrectangle of $\De_i$, whose 
bottom stable side is $J_{i,\rho}$ (see figure \ref{section76}). 
\begin{figure}[h!] 
\includegraphics[scale=0.80]{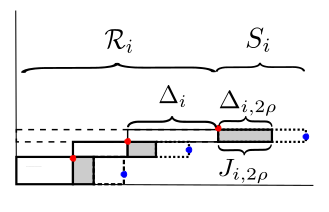}
\caption{}
\label{section76}
\end{figure}

\begin{clai} For any $n>0$ and  $\epsilon>0$ small enough, any $f_A$ periodic orbit $\gamma_+\subset\La_+$  which is $\varepsilon$-dense (in $\La_+$), has period greater than $n$ and has more than $n$ points in $\De_{i,\rho}$ for any $i$.  
\end{clai}
\begin{proof} Up to the action of the group generated by $A$ and the integer translations, we just need to check the claim on finitely many primitive  $\sigma_+$-rectangles. Thanks to Remark~\ref{r.boundary}, we asked that $\sigma_+$ be accumulated by points in $\La_+$ in each quadrant, so the interior of $\De_{i,\rho}$ intersects $\La_+$.  The $\varepsilon$-density implies that for $\varepsilon$ small enough $\De_{i,\rho}$ also contains arbitrarily many points in $\tilde \gamma_+$; the period of $\gamma_+$ is clearly greater than this number.  
\end{proof}

Using the above notations and hypotheses, the aim of this section is to prove
\begin{prop}\label{p.undertwisted} If $n>0$ is large enough then for any vector field $Y \in\cS urg(X,\cE\cup\La_+\cup\La_-)$ for which the characteristic numbers of the surgeries along $\La_+$ are non-positive and equal to $m<0$ on $\gamma_+$, the quadrant $C_{+,+}(x)$ ($x\in \tilde{\sigma_+}$) is incomplete (undertwisted)
\end{prop}

Notice that (for the first time in this paper) we perform a surgery on a periodic orbit $\gamma_+$ and we calculate the holonomy on a quadrant at a different periodic point $\sigma_+$. 

\begin{proof}[Proof of Proposition~\ref{p.undertwisted}]

Proposition~\ref{p.undertwisted} is a direct consequence of the next lemma
\begin{lemm}\label{l.undertwisted} Under the hypotheses of Proposition~\ref{p.undertwisted}, if we denote $J_i=[q_i, t_i]^s$, where  $t_i$ is the right endpoint of $J_{i,2\rho}$, then for $n$ large enough, the unstable holonomy $h^u_{Y,q_i,q_{i+1}}\colon F^s_+(q_i)\to F^s_i(q_{i+1})$ satisfies one of the two following:

\begin{itemize}
 \item $h^u_{Y,q_i,q_{i+1}}(t_i)$ is not defined (so $C_{+,+}(x)$ is incomplete)
 \item $t_{i+1}\subset [q_{i+1},h^u_{Y,q_i,q_{i+1}}(t_i)]^s$. 
\end{itemize}
\end{lemm}
Proposition~\ref{p.undertwisted} follows because the length of $[q_i,t_i]^s$ tends to infinity as $i\to\infty$, hence the unstable holonomy from $F^s_+(x)$ to $F^s_+(q)$ (where $q=\lim q_i$ is the endpoint of $\cI^s(\cR)$) is not defined at $t_0$.  

We proceed to the proof of Lemma \ref{l.undertwisted}. Consider the holonomy game in $R_i \cup S_i$ : starting at a point $t\in F^s_+(q_i)$ let us follow its positive unstable leaf $F^u_+(t)$.  As the rectangle $R_i \cup S_i$ is disjoint from $\tilde \cE\cup \tilde\La_-$ and as the surgeries on periodic orbits in $\La_+$ have non-negative characteristic numbers all holonomies are expansions as long as the point remains inside $R_i\cup S_i$. 

If, while following the holonomy, we exit $R_i\cup S_i$ before reaching $F^s_+(q_{i+1})$ it's impossible to enter back later in the game: as a consequence, either the holonomy $h^u_{Y,q_i,q_{i+1}}$ is not defined at $t$ or $t_{i+1}\in[q_{i+1},h^u_i(t)]$, by our choice of $\rho$. 

Therefore, we just need to check that the point $t_i$ exits $R_i\cup S_i$ before reaching $F^s_+(q_{i+1})$. As all the holonomies that affect it inside $R_i\cup S_i$ are all expansions, it is enough to prove that $t_i$ exits $R_i\cup S_i$ before reaching $F^s_+(q_{i+1})$ only thanks to the points in $\gamma_+\cap \De_{i,\rho}$. Each time the unstable manifold of $t_i$ crosses the stable manifold of one of these points, the distance to this point is multiplied by $\lambda^{\pi(\gamma_+)|m|}\geq \lambda^n$.  This distance is at least $\rho \ell^s(\De_i)$ which is in bounded ratio 
with $\ell^s(R_i)+\ell^s(S_i)$. 

In order to get the desired property, it is enough to choose $n$ such that for every $i$ one has 

$$\lambda^n>\frac{\ell^s(R_i)+\ell^s(S_i)}{\rho\ell^s(\De_i)},$$
which concludes the proof of the lemma, and hence of Proposition~\ref{p.undertwisted}. 

\end{proof}

\subsubsection{Concluding the proof of Theorem~\ref{t.notRcoveredatall}}

Now the proof of Theorem~\ref{t.notRcoveredatall} just consists in applying Proposition~\ref{p.undertwisted} in the quadrants $C_{+,+} (x_+)$ and $C_{-,-}(x_-)$, where $x_+\in\sigma_+$ and $x_-\in\sigma_-$ for a common choice of a small $\varepsilon$ and of orbits $\gamma_+\subset \La_+$ and $\gamma_-\subset \La_-$, which are $\varepsilon$-dense in $\La_+$ and $\La_-$ respectively.

\begin{rema}Theorem~\ref{t.notRcoveredatall} only announces the existence of a pair of orbits $\gamma_+$ and $\gamma_-$.  In the previous proof, we've established slightly more, that is, for any two disjoint basic sets $\La_+$ and $\La_-$, which are also disjoint from the arbitrary given set $\cE$, there is $\epsilon >0$ such that the Theorem~\ref{t.notRcoveredatall} holds for any $\gamma_+\subset \La_+$ and $\gamma_-\subset \La_-$ which are $\varepsilon$-dense in $\La_+$ and $\La_-$ respectively. 
 
\end{rema}
\section{Surgeries on $2$ periodic orbits}\label{s.deux}

The aim of this final section is to give an overview of the vector fields obtained from a suspension flow $X_A$ by performing surgeries on exactly $2$ periodic orbits.  In other words, using the the previous notations, $\cX$ and $\cY$ are each a single periodic orbit.

There are, in theory, sixteen different cases, according to the existence or non-existence of positive or negative, $\cX$-rectangles disjoint from $\tilde \cY$ or $\cY$-rectangles disjoint from $\tilde \cX$. 

However, Lemma~\ref{l.asymetric} implies that if there are no positive $\cX$-rectangles disjoint from $\tilde \cY$ then there are negative $\cY$-rectangles disjoint from $\tilde \cX$. 

Therefore, up to interchanging positive and negative or $\cY$ and $\cX$, there are the following $4$ cases which will be examined separately in the next pages. 

\begin{enumerate}
\item there are positive and negative $\cX$-rectangles disjoint from $\tilde \cY$ and $\cY$-rectangles disjoint from $\tilde \cX$. 
\item there are no $\cX$-rectangles (neither positive or negative) disjoint from $\tilde \cY$ (so there are positive and negative $\cY$-rectangles disjoint from $\tilde \cX$). 
\item there are no negative $\cX$ rectangles disjoint from $\tilde \cY$ and no negative $\cY$-rectangles disjoint from $\tilde \cX$. 
\item there are no positive $\cX$ rectangles disjoint from $\tilde \cY$, but there are rectangles in the $3$ other categories. 
\end{enumerate}

In each case we will consider the vector field $Z_{m,n}$ obtained by an $(m,n)$ surgery along $\cX$ and $\cY$ and discuss what we know about the bi-foliated plane of $Z_{m,n}$, according to the position of $(m,n)$ in the lattice $\ZZ^2$.

\subsection{Case 1: Assuming the existence of separating rectangles of all types}

Consider $\ZZ^2$ and  the vector field $Z_{m,n}$ obtained by an $(m,n)$ surgery along
$\cX$ and $\cY$. 

Then 
\begin{itemize}
 \item if $n,m$ have the same sign (or one of them vanishes) then  $Z_{m,n}$ is $\RR$-covered twisted in the  direction of that sign, according to \cite{Fe1}
 
\item if $n,m$ have opposite signs and are large enough, then $Z_{m,n}$ is not $\RR$-covered according to Theorem \ref{t.string}. 

\item if $n,m$ have opposite signs and one of them is large enough, then using lemma \ref{l.exists staircase} some quadrant is incomplete: $Z_{n,m}$ is either non $\RR$-covered or twisted. 

\end{itemize}

This case can be realised by considering sets periodic points following the homoclinic intersections of two fixed periodic points, as it was already done in the proof of Corollary \ref{c.notRcoveredatall}. In many cases it could be checked that only the behaviours $1$ and $2$ are possible. 

\subsection{Case 2: no $\cX$-rectangle disjoint from $\tilde \cY$}

This case can be realised as follows: take any $\cX\subset \TT^2$ and choose $\varepsilon >0$ small enough such that any $\varepsilon$-dense periodic orbit $\cY$ intersects the interior of every $\cX$-rectangle.

In this case,
\begin{itemize}
 
\item if $n$ is large enough in absolute value then using Theorem \ref{t.epsilon}, $Z_{m,n}$ is $\RR$-covered twisted according to the sign of $n$.  
\item if $n$ vanishes, then $Z_{m,0}$ is $\RR$-covered twisted according to the sign of $m$ (\cite{Fe1})

\end{itemize}

\subsection{case 3: no negative $\cX$-rectangles disjoint from $\tilde \cY$ and no negative $\cY$-rectangles disjoint from $\tilde \cX$}
Contrary to the previous cases, we don't have canonical examples, where this case is realised. Nevertheless, one could check that if $\cX=(0,0)$, $\cY=(0,1/2)$ and 
$$A=\left(\begin{array}{cc} 
            3&2\\
            4&3
           \end{array}\right)$$
(see figure \ref{norectangles}) there are no negative $\cX$-rectangles disjoint from $\tilde \cY$ and no negative $\cY$-rectangles disjoint from $\tilde \cX$. Our proof of this fact consists in understanding the nature of the continued fractions associated to the slopes of the eigendirections and thus goes beyond the purposes of this paper.  
In this case,
\begin{itemize}
 
\item if $n$ or $m$ is negative and large enough in absolute value then using Theorem \ref{t.twisted}, $Z_{m,n}$ is $\RR$-covered twisted negatively.  
\item if $n$ and $m$ have the same sign then $Z_{m,n}$ is $\RR$-covered twisted according to the sign of $m$ or $n$ (\cite{Fe1})

\end{itemize}
\begin{figure}[h!] 
\includegraphics[scale=0.80]{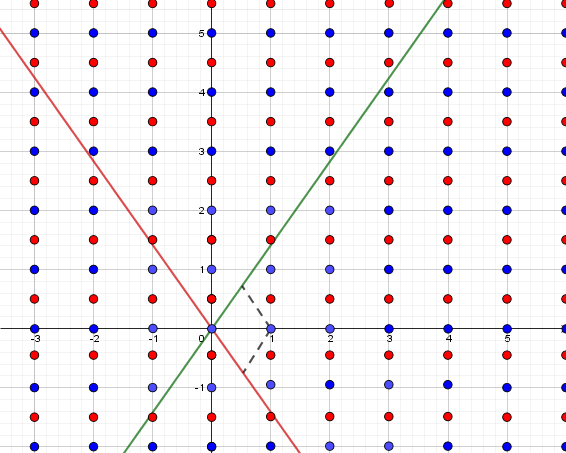}
\caption{In the above figure red points represent lifts of the point $(0,\frac 12)$, blue points lifts of $(0,0)$, the red line is the stable eigendirection and the green one the unstable}
\label{norectangles}
\end{figure}
\subsection{case 4: no positive $X$-rectangles disjoint from $\tilde \cY$, but all other rectangles}
We haven't been able to come up with an example, satisfying the hypotheses of this case, but it seems possible to us that an example similar to the one of case 3 makes this case also realisable. 
\section{Explicit examples}\label{s.explicit}

In this section we consider more specifically, the orbits of $(0,0)$ and $(\frac12,\frac12)$. For any $A\in SL(2,\ZZ)$ the point $(0,0)$ is a fixed point of $f_A$, but for the point $(\frac12,\frac12)$ there are $3$ possibilities: 
\begin{itemize}
 \item either $(\frac12,\frac12)$ is a fixed point
 \item or $(\frac12,\frac12)$ is a periodic point of period $2$
 \item or it is a periodic point of period $3$, whose orbit is exactly
 $$\{(0,\frac12),(\frac12,0),(\frac12,\frac12)\}.$$
\end{itemize}

For instance $(\frac12,\frac12)$ is a periodic of period $3$ (resp. $2$) for every matrix of the form
$$A_k=\left(\begin{array}{cc} 
            k&k-1\\
            1&1
           \end{array}\right)$$
           with $k\in 2\NN^*$ (resp. $k\in 2\NN+3$)

\begin{rema}
 Given any matrix $A\in SL(2,\ZZ)$, any positive or negative $(0,0)$-rectangle contains a point of $\{(0,\frac12),(\frac12,0),(\frac12,\frac12)\}+\ZZ^2$
\end{rema} 
 Indeed, a $(0,0)$-primitive rectangle does not contain any other integer points and has a diagonal whose endpoint is an integer point.  Therefore, the middle point of that diagonal cannot be an integer point, hence it  belongs to $\{(0,\frac12),(\frac12,0),(\frac12,\frac12)\}+\ZZ^2$. Using the previous remark and by applying Theorem \ref{t.twisted} we have the following result:

\begin{coro}Given any matrix $A\in SL(2,\ZZ)$, consider any vector field $Y$ obtained from $X_A$ by performing surgeries on the orbits corresponding to the set $\{(0,0), (0,\frac12),(\frac12,0),(\frac12,\frac12)\}$ such that the characteristic numbers associated to the points $\{(0,\frac12),(\frac12,0),(\frac12,\frac12)\}$ have the same sign $\omega\in\{+,-\}$ and are large enough. 

Then $Y$ is $\RR$-covered and $\omega$-twisted. 
\end{coro}
Also by our above remark, the triples $(X_{A_k},(0,0),\{(0,\frac12),(\frac12,0),(\frac12,\frac12)\})$ with $k\in 2\NN^*$ provide infinitely many examples that realise the case (2) of Section~\ref{s.deux}. 

Consider now the matrix $B_k=A_k^3$ when $k\in 2\NN^*$ and $B_k=A_k^2$ when $k\in 2\NN+3$. 

\begin{lemm}For any $k$, the Anosov map $f_{B_k}$ admits positive and negative $(0,0)$-rectangles disjoint from $\widetilde{(\frac12,\frac12)}$ and positive and negative $(\frac12,\frac12)$-rectangles disjoint from $\widetilde{(0,0)}$.
 
\end{lemm}

\begin{proof} Notice that the foliations of $A_k$ and $B_k$ coincide.  We denote  
$$F^s_k=F^s_{B_k}=F^s_{A_k}\mbox{ and } F^u_k=F^u_{B_k}=F^u_{A_k}.$$ 
Because $A_k$ has positive coefficients its unstable direction is inside $C_{+,+}\cup C_{-,-}$ and its stable direction in $C_{+,-}\cup C_{-,+}$. 

By looking at the image of the $C_{+,+}$ quadrants one gets that the unstable direction $E^u$ is between the increasing (usual) diagonal of $\RR^2$ and the $x$-axis. In the same way, by looking at the inverse image of the $C_{+,-}$ quadrant, one checks that the stable direction $E^s$ is between the decreasing diagonal and the $y$-axis. 

One deduces by the previous observations that the $(0,0)$-rectangle admitting $[0,1]\times \{0\}$ as a diagonal is a positive primitive $(0,0)$-rectangle disjoint from $\widetilde{(\frac12,\frac12)}$. 
In the same way, the $(0,0)$-rectangle admitting $\{0\}\times [0,1]$ as a diagonal is a negative primitive $(0,0)$-rectangle disjoint from $\widetilde{(\frac12,\frac12)}$.

Finally, the translated by $(\frac12,\frac12)$ positive and negative $(0,0)$-rectangles disjoint from $\widetilde{(\frac12,\frac12)}$ are respectfully positive and negative $(\frac12,\frac12)$-rectangles disjoint from $\widetilde{(0,0)}$, which ends the proof. 
 
\end{proof}

The triples $(X_{B_k},(0,0),(\frac12,\frac12))$ provide infinitely many examples that realise the situation (1) of Section~\ref{s.deux}.

\vspace{.5cm}

{\bf Christian Bonatti} bonatti@u-bourgogne.fr\\ {\bf Ioannis Iakovoglou}ioannis.iakovoglou@ens-lyon.fr\\  Universit\'{e} de
Bourgogne,\\ Institut de Math\'{e}matiques de Bourgogne, UMR 5584 du
CNRS, BP 47 870,\\ 21078, Dijon Cedex, France. \vspace{.5cm}

\end{document}